\documentclass[11pt]{amsart}  %
\usepackage{amssymb}
\usepackage{amsmath}
\usepackage{fancybox}
\usepackage{graphics}
\usepackage{latexsym}
\usepackage[applemac]{inputenc}
\usepackage{ae,aecompl}
\usepackage{amsthm}
\usepackage{mathrsfs}
\usepackage{amsfonts,amssymb}
 \usepackage[pdftex]{graphicx}
  \usepackage[colorlinks=true]{hyperref}
\usepackage{xcolor}

\DeclareGraphicsExtensions{.jpg,.pdf,.eps}
\graphicspath{{./Figures/}}

\usepackage{color}

\usepackage{paralist}

\newcommand{\fig}[3]{\begin{figure}[h]\begin{center}\includegraphics[#1]{#2}\end{center}\caption{#3}\label{fig:#2}\end{figure}}
\newcommand{\ds}{\displaystyle}

\newcommand{\res}{\textrm{Resultant}}

\newcommand{\al}{\alpha}
\newcommand{\be}{\beta}

\newcommand{\vv}{\mathrm{v}}
\newcommand{\vopen}{\textrm{v}_\bullet}
\newcommand{\vclosed}{\textrm{v}_\circ}
\newcommand{\vinact}{\textrm{v}_{\textrm{inact}}}
\newcommand{\vact}{\textrm{v}_{\textrm{act}}}
\newcommand{\eact}{\textrm{e}_{\textrm{act}}}
\newcommand{\eopen}{\textrm{e}_\bullet}
\newcommand{\eclosed}{\textrm{e}_\circ}

\newcommand{\fin}{\textrm{f}_{\textrm{in}}}
\newcommand{\eout}{\textrm{e}_{\textrm{out}}}
\newcommand{\reef}{\textrm{reef}}
\newcommand{\ereef}{\textrm{e}_{\reef}}
\newcommand{\ee}{\mathrm{e}}
\newcommand{\vout}{\textrm{v}_{\textrm{out}}}
\newcommand{\length}{\textrm{length}}
\newcommand{\Ws}{\dot{W}}
\newcommand{\Wb}{\overline{W}}

\newcommand{\mIs}{\dot{\mathcal{I}}}
\newcommand{\mIe}{\overline{\mathcal{I}}}

\newcommand{\mS}{\mathcal{S}}
\newcommand{\mT}{\mathcal{T}}
\newcommand{\mR}{\mathcal{R}}

\newcommand{\hT}{\widehat{T}}
\newcommand{\tT}{\widetilde{T}}
\newcommand{\tW}{\widetilde{W}}
\newcommand{\tS}{\widetilde{S}}
\newcommand{\hmT}{\widehat{\mT}}

\newcommand{\tc}{\tilde{c}}

\newcommand{\R}{\mathbb R}

\newcommand{\Q}{\mathbb Q}

\renewcommand{\P}[2]{\mathbb{P}_{#2}\left( #1 \right)}

\def\bq{{\bf q}}

\def\build#1_#2^#3{\mathrel{
\mathop{\kern 0pt#1}\limits_{#2}^{#3}}}

\newtheorem{theorem}{Theorem}[section]

\newtheorem{proposition}[theorem]{Proposition}
\newtheorem{lemma}[theorem]{Lemma}

\newtheorem{definition}[theorem]{Definition}

\newtheorem{claim}[theorem]{Claim}

\newtheorem{conj}[theorem]{Conjecture}

\renewcommand{\P}[1]{\mathbb{P}\left(#1\right)}
\newtheorem{rek}[theorem]{Remark}

\usepackage{enumerate}

\title{A Boltzmann approach to percolation on random triangulations}
\author{Olivier Bernardi} 
\address{Department of Mathematics, Brandeis University, USA}
\email{bernardi@brandeis.edu}

\author{Nicolas Curien}
\address{D\'epartement de Math\'ematiques de l'Universit\'e Paris-Sud, and Institut Universitaire de France}
\email{nicolas.curien@gmail.com}

\author{Gr\'egory Miermont}
\address{Unit\'e de Math\'ematiques Pures et Appliqu\'ees de l'\'Ecole Normale Sup\'erieure de Lyon, and Institut Universitaire de France}
\email{gregory.miermont@ens-lyon.fr}

\subjclass[2010]{60K35, 
60D05, 
05A16} 

\date{\today}

\begin{document}
\maketitle
\begin{abstract}
We study the percolation model on Boltzmann triangulations using a generating function approach. More precisely, we consider a Boltzmann model on the set of finite planar triangulations, together with a percolation configuration (either site-percolation or bond-percolation) on this triangulation. 
By enumerating triangulations with boundaries according to both the boundary length and the number of vertices/edges on the boundary, we are able to identify a phase transition for the geometry of the origin cluster. 
For instance, we show that the probability that a percolation interface has length $n$ decays exponentially with $n$ except at a particular value $p_c$ of the percolation parameter $p$ for which the decay is polynomial (of order $n^{-10/3}$). Moreover, the probability that the origin cluster has size $n$  decays exponentially if $p<p_c$ and polynomially if $p\geq p_c$. 

The critical percolation value is $p_c=1/2$ for site percolation, and $p_c=\frac{2\sqrt{3}-1}{11}$ for bond percolation. These values coincide with critical percolation thresholds for infinite triangulations identified by Angel for site-percolation, and by Angel \& Curien for bond-percolation, and we give an independent derivation of these  percolation thresholds.

Lastly, we revisit the criticality conditions for random Boltzmann maps, and argue that  at $p_c$, the percolation clusters conditioned to have size $n$ should converge toward the stable map of parameter $ \frac{7}{6}$ introduced by Le Gall \& Miermont. This enables us to derive heuristically some new critical exponents. 
\end{abstract}

\section{Introduction}
The percolation model on random planar maps has been extensively studied in recent years in particular through the peeling process. Indeed, it is often possible to use the spatial Markov property of the underlying lattice to define an exploration along the percolation interface and get access to the percolation threshold. This approach was first developed in the pioneer work of Angel \cite{Ang03} for site-percolation on the Uniform Infinite Planar Triangulation (UIPT) and later extended to other models of percolation and maps \cite{ACpercopeel,CurPeccot,MN13,Ric15}. As opposed to the ``dynamical'' approach of the peeling process, the work \cite{CKperco} uses a ``fixed'' combinatorial decomposition (inspired by \cite{BBG12}) and known enumeration results on triangulations to study the scaling limit of percolation cluster conditioned on having a large boundary. All the above works focused, in a sense, on the geometry of one percolation interface, hence studied the geometry of the outer boundary of a large percolation cluster. The present paper however, genuinely studies the geometry of the \emph{full cluster} of the origin in a finite map. 

Let us give a rough idea of our setting before giving more precise definitions. We consider a \emph{critical Boltzmann triangulation}, that is a random finite planar triangulation $M$ chosen with probability proportional to $z_0^{\#\textrm{triangles}}$, where $z_0=432^{-1/4}$ is the maximal value for which this definition makes sense. Under this law, the probability that $M$ has $n$ triangles decays polynomially in $n$. We then endow $M$ with a Bernoulli bond or site percolation model with parameter $p\in [0,1]$, and consider the origin cluster $\mathfrak{C}(p)$. The cluster $\mathfrak{C}(p)$ is a random planar map which also has a Boltzmann distribution, in the sense that there is a sequence $(q_k)_{k>0}$ of non-negative numbers depending on the parameter $p$, such that the probability that $\mathfrak{C}(p)$ is equal to any map $\mathfrak{m}$ is proportional to the product over  all faces $f$ of $ \mathfrak{m}$ of $q_{\deg(f)}$ (see below). We show that there is a phase transition of the percolation model at a certain \emph{critical value} $p=p_c$ (with  $p_c=1/2$ for site-percolation, and $p_c=\frac{2\sqrt{3}-1}{11}$ for bond-percolation). This phase transition manifests itself in at least three ways:
\begin{itemize}
\item[(a)] the probability that the cluster $\mathfrak{C}(p)$ has $n$ vertices decays exponentially in $n$ for $p<p_c$ and polynomially for $p\geq p_c$,
\item[(b)] the probability that the percolation interface surrounding $\mathfrak{C}(p)$ has length $\ell$ decays exponentially in $\ell$ for $p\neq p_c$ and polynomially for $p=p_c$,
\item[(c)] the asymptotic form of the sequence  $(q_k)_{k>0}$ is different for  $p< p_c$, $p=p_c$ and $p>p_c$.
\end{itemize}
The result (a) is closely related to the usual definition of the critical percolation threshold on infinite graphs (the infimum of the $p$'s for which the origin cluster can be infinite). We indeed establish a link between our critical values of $p_c$ and the critical percolation thresholds previously obtained for percolation on the \emph{uniform infinite planar triangulation} (UIPT) of Angel and Schramm \cite{AS03} and its half-plane analog. The result (b) indicates that the critical cluster $\mathfrak{C}(p_c)$ conditioned to have many vertices will have some faces of polynomially large degrees. The result (c) allows us to show that  the critical cluster $\mathfrak{C}(p_c)$ is a \emph{non-regular critical Boltzmann map} in the sense of Le Gall and Miermont \cite{LGM09}. It strongly suggests (although we do not attempt to prove this) that the rescaled critical percolation cluster conditioned to have $n$ vertices converges in law toward the so-called \emph{stable map} of parameter $ \frac{7}{6}$. This conjectural limit leads us to make several additional conjectures on the geometry of $\mathfrak{C}(p_c)$.

\paragraph*{Boltzmann maps and percolated triangulations.} 
We will now give more precise definitions, and state our main results.
We use the standard terminology for planar maps, see Section \ref{sec:defmaps} for precise definitions. In this article, all our maps are \emph{planar} and \emph{rooted}. 
Following \cite{MM07,Mie08b}, given a (non-zero) sequence of non-negative weights $ \mathbf{q}= (q_{k})_{k \geq 1}$, we define the \emph{$ \mathbf{q}$-Boltzmann measure} $ \mathrm{Bolt}_{ \mathbf{q}}$ on the set of finite (rooted planar) maps by the formula:
 \begin{eqnarray} \mathrm{Bolt}_{ \mathbf{q}}( \mathfrak{m}) &=& \prod_{ f \in \mathsf{Face}( \mathfrak{m})} q_{ \mathrm{deg}(f)}. \label{eq:boltzmann} \end{eqnarray} 
When the total mass of $ \mathrm{Bolt}_{ \mathbf{q}}$, 
 \begin{equation}
  \label{eq:1}
Z_\bq=  \sum_{\mathfrak{m}\textrm{ rooted planar map}}\mathrm{Bolt}_{ \mathbf{q}}(\mathfrak{m})<\infty\, ,
 \end{equation}
is finite, we say that $ \mathbf{q}$ is \emph{admissible}, and we can then renormalize $ \mathrm{Bolt}_{ \mathbf{q}}$ into a probability measure that we call the $ \mathbf{q}$-Boltzmann probability distribution. The usual definition of admissibility in \cite{Mie08b} requires the apparently stronger condition
 \begin{equation}
  \label{eq:2}
 Z^\bullet_\bq= \sum_{\mathfrak{m}} \mathrm{Bolt}_{ \mathbf{q}}(\mathfrak{m})\, \vv(\mathfrak{m})<\infty\, ,
 \end{equation}
where $\vv(\mathfrak{m})$ is the number of vertices of $\mathfrak{m}$, 
 although \eqref{eq:1} and \eqref{eq:2} turn out to be equivalent, as we will see in Proposition \ref{sec:char-admiss-crit}. We say that the admissible weight sequence $\bq$ is {\em critical} if 
 \begin{equation}
  \label{eq:11}
  \sum_{\mathfrak{m}} \mathrm{Bolt}_{ \mathbf{q}}( \mathfrak{m})\,  \vv(\mathfrak{m})^{2 } = \infty\, ,
 \end{equation}
 and subcritical otherwise. 
We will see in Section \ref{sec:critical} that this definition coincides with the original one in \cite{Mie08b}, which will be recalled in due time.

A particular case of weight sequence is given by $ \mathbf{q}= ( z \times \delta_{k,3})_{k \geq 1}$ so that the associated Boltzmann measure gives a weight $z^{n}$ to each triangulation (type-I where loops and multiple edges are allowed) with $n$ faces, and weight zero to any other map. By a classical result of Tutte \cite{Tutte:census3} we have 
$$ \# \{ \mbox{triangulations with }n \mbox{ faces}\} \underset{n \to \infty}{\sim} c_{0} \sqrt[4]{432}^{n} n^{-5/2},$$ for some $c_{0} >0$ and so the last weight sequence is admissible if and only if $z \leq 1/ \sqrt[4]{432}$. For $z = z_{0}:= {1}/{ \sqrt[4]{432}}$ the weight sequence, denoted below by $ \mathbf{q}_{0}$, is furthermore critical (and subcritical if $z < z_{0}$) and we call the renormalized measure the \emph{critical Boltzmann measure on triangulations}. 

For $p \in [0,1]$, under the critical Boltzmann measure on
triangulations, we perform a site (resp.~bond) percolation on the
underlying triangulation $ M$ by independently coloring each vertex
(resp.~edge) of $ M$ in black with probability $p$ and in white with
probability $1-p$. On the event that the root edge is colored
in black in the case of bond-percolation, or that its endpoints are
colored black in the case of site-percolation, we consider the map $ \dot{\mathfrak{C}}(p)$ (in the case of site-percolation) or $ \overline{\mathfrak{C}} (p)$ (in the case of bond-percolation) made of the black cluster of the origin in the percolated triangulation, naturally rooted at the same edge as $M$ (see Figure~\ref{fig:clusters}). 
In the case where the root edge of $M$ is not colored black (which in the case of site percolation means that at least one extremity of the root edge is colored white), then by convention we let $\dot{\mathfrak{C}}(p),\overline{\mathfrak{C}}(p)$ be the atomic map, with only one vertex and no edge. 

To state our theorem in a condensed form, let us say that a sequence $ \mathbf{u}=(u_{k})_{k \geq 1}$ of non-negative numbers is \emph{orthodox} with \emph{growth constant} $R >0$ and \emph{exponent} $\beta \in \mathbb{R}$ if for some constant $c>0$ we have 
 $$ u_{k} \quad \underset{k \to \infty}{\sim} c \times R^{k} \times k^{-\beta}.$$

\begin{theorem}[Main result] \label{thm:main} For any $p \in [0,1]$
  and under the critical Boltzmann measure on triangulations, conditionally on the event that the root edge is colored black, both
  random maps $\dot{\mathfrak{C}}(p)$ and $
  \overline{\mathfrak{C}}(p)$ are Boltzmann distributed with admissible orthodox weight sequences $ \dot{\mathbf{q}}(p)$ and $ \overline{\mathbf{q}}(p)$ for $p \in [0,1]$, and conditioned on having at least one edge. If we set $$ \dot{p}_{c} = \frac{1}{2} \quad \mbox{ and } \quad \overline{p}_{c}= \frac{2 \sqrt{3}-1}{11}$$ then the exponents $\dot{\beta}(p)$ and $ \overline{\beta}(p)$ of $\dot{\mathbf{q}}(p)$ and $ \overline{\mathbf{q}}(p)$ satisfy:

 \renewcommand{\arraystretch}{1.8}
$$
\begin{array}{|c|c|c|c|}
\hline
 & p \in [0, \dot{p}_{c}) & p= \dot{p}_{c}& p \in (\dot{p}_{c},1]\\
 \hline
 \dot{\be}(p) & 5/2 & \displaystyle 5/3 & 3/2\\
 \hline
 \end{array} \qquad \begin{array}{|c|c|c|c|}
\hline
 & p \in [0, \overline{p}_{c}) & p= \overline{p}_{c}& p \in (\overline{p}_{c},1]\\
 \hline
 \overline{\be}(p) & 5/2 & \displaystyle 5/3 & 3/2\\
 \hline
 \end{array}
 $$
Furthermore, for $p < \dot{p}_{c}$ (resp.~$p< \overline{p}_{c}$) the distribution of the Boltzmann map $ \dot{ \mathfrak{C}}(p)$ (resp.~$\overline{ \mathfrak{C}}(p)$) is subcritical, and the probability that this map has size $n$ decreases exponentially with $n$. For $p\geq  \dot{p}_{c}$  (resp.~$p\geq \overline{p}_{c}$) the distribution of the Boltzmann map $ \dot{ \mathfrak{C}}(p)$ (resp.~$\overline{ \mathfrak{C}}(p)$) is critical, and the probability that this map has size $n$ decreases polynomially with $n$. 
\end{theorem}

 \begin{figure}[!h]
 \begin{center}
 \includegraphics[width=10cm]{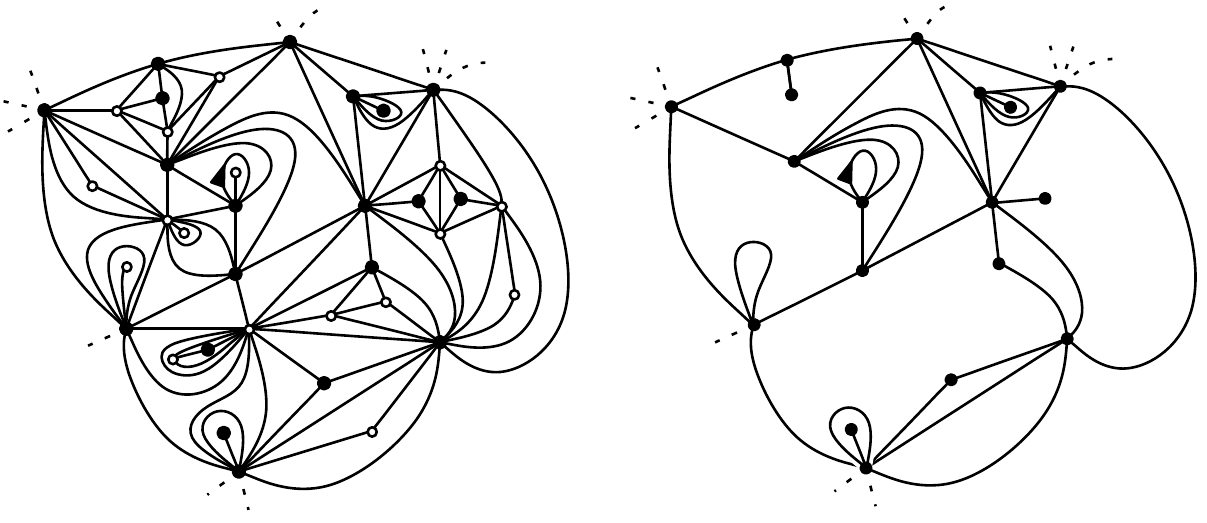} 
  \includegraphics[width=10cm]{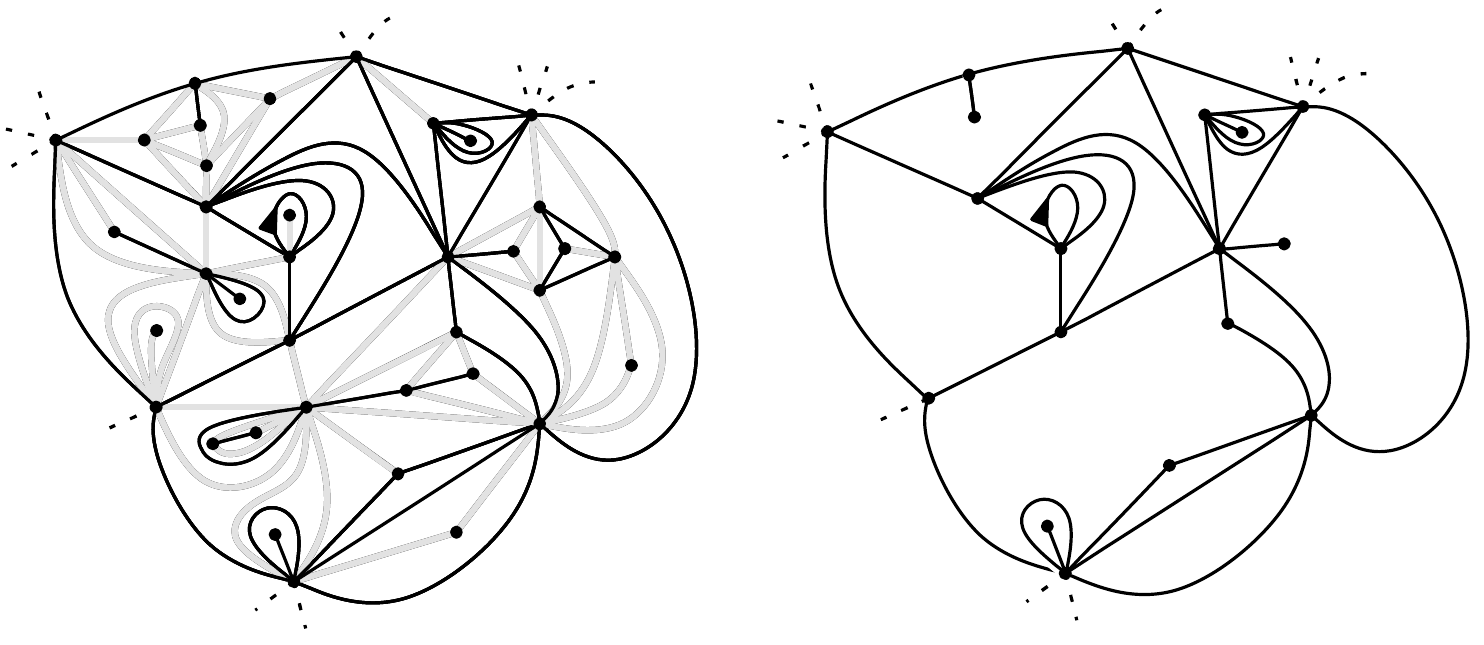}
 \caption{ \label{fig:clusters} Left column: a piece of a percolated triangulation (site-percolation on the first row, bond-percolation on the second row). Right column: the resulting black cluster of the origin.}
 \end{center}
 \end{figure}

Roughly speaking, the above theorem (which follows from our Propositions \ref{prop:weight-asymptotic-site}, \ref{prop:weight-asymptotic-bond}, \ref{prop:interface-site} and \ref{prop:interface-bond} below) indicates  a phase transition for the geometry of the origin cluster of percolated critical Boltzmann triangulations: for $p<p_c$ the origin cluster is ``small'', while for $p\geq p_c$ this cluster may be ``large''. 
We recover in this result the particular role played by the critical values $ \dot{p}_{c} = \frac{1}{2}$ and $ \overline{p}_{c}= (2 \sqrt{3}-1)/11$ which had already been identified as the almost sure critical percolation thresholds for site and bond-percolations on infinite random triangulations, see \cite{Ang03} for the case of site-percolation on the UIPT and \cite{ACpercopeel} for site and bond percolations on the half-planar version of the UIPT. Notice also that the value $ \dot{p}_{c}= \frac{1}{2}$ is also pivotal in the work \cite{CKperco} dealing with scaling limit of cluster boundaries on the UIPT. This is of course not surprising, and our work furnishes an independent proof that $ \dot{p}_{c}$ and $ \overline{p}_{c}$ are the percolation thresholds for site and bond percolation on the UIPT (a result which is new in the case of bond percolation) together with a proof of exponential decay of the cluster size in the subcritical phase:

\begin{theorem}[Percolation on the UIPT] \label{thm:UIPT}
 \label{sec:links-with-perc-1}
The (almost sure quenched) percolation thresholds for site and bond percolations on the Uniform Infinite Planar Triangulation are given by 
$$\dot{p}_c(\mathrm{UIPT})=\dot{p}_c=\frac{1}{2}\, ,\qquad \overline{p}_c(\mathrm{UIPT})=\overline{p}_c=\frac{2\sqrt{3}-1}{11}\, ,$$
and moreover, in the subcritical case $p<\dot{p}_c$ (resp.\ $p<\overline{p}_c$), the tail distribution of the number of vertices in the origin site (resp.~bond) percolation cluster decays exponentially.
\end{theorem}

Our main result should also imply that the large scale geometry of the critical percolation clusters $ \dot{\mathfrak{C}}(\dot{p}_{c})$ and $ \overline{ \mathfrak{C}} ( \overline{{p}}_{c})$ are described by the stable maps\footnote{For the connoisseur, note that the uniqueness of the stable maps is still an open problem and so, as in \cite{LGM09}, we would need to pass to a subsequence to establish scaling limits results.} of parameter\footnote{In the notation of \cite{LGM09} we have $\alpha = 7/6$ as well as $a= {5}/{3}$ so that $7/6=\alpha = a -1/2$.} $ 7/6$ introduced in \cite{LGM09}. Unfortunately, the work \cite{LGM09} only deals with \emph{bipartite} planar maps whereas our clusters are non necessarily bipartite random maps. However, performing a leap of faith we proceed in Section \ref{sec:stablemaps} to the non-rigorous derivation of several critical exponents based on the approach of \cite{LGM09}.

\paragraph{A word on the proofs.} As mentioned above, our approach is based first on a combinatorial decomposition of percolated triangulations (Section \ref{sec:islands-gene}), which, roughly speaking, enables us to decouple between the cluster of the origin and the ``islands'' it splits in the map. This directly entails that the clusters $ \dot{ \mathfrak{C}}$ and $ \overline{ \mathfrak{C}}$ are Boltzmann distributed with weights related to the Boltzmann weight of the ``islands''. After a further reduction, these weights are computed using a generating function approach ``\`a la Tutte'' and solved using the methods pioneered by Bousquet-M\'elou and Jehanne \cite{MBM:quadratic}. 
For the site-percolation model, this boils down to the enumeration of triangulations with boundary according to the number of outer vertices. For the bond-percolation model, this boils down to the enumeration of triangulations with simple boundary according to the number of edges incident to outer vertices. These calculations, which are the core of the present work, are performed in Section \ref{sec:gf} and eventually yields the asymptotic form of the weight sequences presented in Theorem \ref{thm:main}.

Most of the enumeration results in the case of site-percolation could be derived from the work \cite{CKperco} (see Remark \ref{rek:CK}) however our angle here is different since we use generating functions and analytic combinatorics methods as opposed to purely probabilistic arguments (Galton--Watson trees and local limit theorems) in \cite{CKperco}. This also shows the robustness of the present approach which also works for bond-percolation. 

As proved in Section \ref{sec:critical}, the criticality or subcriticality of the origin clusters mentioned in Theorem \ref{thm:main} are consequences of the form of the exponents provided in Theorem \ref{thm:main}. This may be surprising at first glance since the (sub)criticality condition \cite{Mie08b} is an \emph{exact} condition on the weight sequence $ \mathbf{q}$ and in particular can not be granted only by an asymptotic on the $q_{k}$'s. However, as noticed in \cite{BBG12} in a slightly different context, the weight sequence $\dot{ \mathbf{q}}(p)$ and $ \overline{ \mathbf{q}}(p)$ also encode an exact information about the Boltzmann measure since the weight $q_{k}$ is closely related to the so-called disk partition function $ \mathrm{Bolt}_{ \mathbf{q}}( \mathcal{M}^{(k)})$ where $ \mathcal{M}^{(k)}$ is the set of all maps of perimeter $k$. Using this precise link as well as our Proposition \ref{prop:carac-critical} we are able to deduce the criticality condition only based on the asymptotic of the weight sequence.

Finally our results are transferred to the case of the UIPT using local absolute continuity relations and the exponential decay of the cluster size in the subcritical regime, see Section \ref{sec:links-with-perc}. 

\medskip 

\textbf{Acknowledgments:} We thank the Newton institute for hospitality during the Random Geometry program in 2015 where part of this work was completed. We acknowledge the support of the NSF grant DMS-1400859, and of the Agence Nationale de la Recherche via the grants ANR Liouville (ANR-15-CE40-0013) and ANR GRAAL (ANR-14-CE25-0014).

 \tableofcontents
 
\section{Percolation models and island decomposition}
\label{sec:islands-gene}
In this section we recall some basic definition about planar maps. We then define the island decomposition which enables us to decouple between the origin black cluster and the ``islands'' that it cuts out from the percolated map.
\subsection{Maps}
\label{sec:defmaps}
A \emph{planar map} (or \emph{map} for short) is a proper embedding of a finite connected graph in the two-dimensional sphere, considered up to orientation-preserving homeomorphisms of the sphere. The \emph{faces} of the map are the connected components of the complement of edges, and the \emph{degree of} a face is the number of edges that are incident to it, with the convention that if both sides of an edge are incident to the same face, this edge is counted twice. A \emph{corner} is the angular section between two consecutive edges around a vertex. Note that the degree of a face or vertex is the number of incident corners. 

As usual in combinatorics, we will only consider \emph{rooted} maps that are maps with a distinguished oriented edge, called \emph{root edge}. The origin of the root edge is called the \emph{root vertex}. The face at the right of the root edge is called \emph{root face}. The corner following the root edge clockwise around the root vertex is called \emph{root corner}. 
Note that the oriented root edge is uniquely determined by the root corner, and in figures we will sometime indicate the rooting of our map by drawing an arrow pointing to the root corner.
We call \emph{atomic map} the rooted map with one vertex and no edge (it still has a root corner). For a rooted map, the vertices and edges incident to the root face are called \emph{outer} and the other vertices and edges are called \emph{inner}. 

A \emph{triangulation} is a (rooted) planar map whose faces are all triangles, that is, have degree three. 
We call \emph{triangulation with boundary} (of length $k$) a rooted planar maps where every non-root face has degree 3 (and the root face has degree $k$). 
It is a triangulation with \emph{simple boundary} if the outer edges form a simple cycle. 
We denote by $\mT$ the set of triangulations with boundary; by convention it includes the atomic map. We denote by $\mS$ the set of triangulations with simple boundary; by convention it does not include the atomic map (so that the boundary length is at least 3).

\subsection{Decomposition for site-percolation}
Let $ \mathfrak{t}$ be a site-percolated triangulation of the sphere. Recall our convention that the endpoints of the root edge must be colored in black. The origin cluster $ \dot{\mathfrak{C}}$ is the planar map obtained by keeping only those edges of the map $ \mathfrak{t}$ whose endpoints are in the black cluster of the root edge (this map is obviously rooted at the root edge of $ \mathfrak{t}$). 
\subsubsection{Isolating the islands}
\label{sec:island-site}
Clearly, the origin cluster $ \dot{\mathfrak{C}}$ may not be a triangulation anymore and its faces could be of two types: either an original face of $ \mathfrak{t}$, or a union of several faces of $ \mathfrak{t}$ that surround some white vertices of $ \mathfrak{t}$. 
By cutting along both sides of the edges of $ \dot{\mathfrak{C}}$, the interior of each face of $ \dot{\mathfrak{C}}$ gets separated into a map that we call \emph{site-island}; see Figure~\ref{fig:cutting-site-island}. We now give a more precise characterization of the type of maps we obtain by this decomposition.
\begin{definition} A \emph{site-island} is a triangulation with simple boundary
together with a site-percolation configuration such that 
\begin{compactitem}
\item[(i)] all the outer vertices are black,
\item[(ii)] all the inner edges incident to an outer vertex are also incident to a white inner vertex.
\end{compactitem}
Examples are given in Figure~\ref{fig:cutting-site-island} below and Figure~\ref{fig:necklace2} (left).\end{definition}

\begin{figure}[!h]
 \begin{center}
 \includegraphics[height=5cm]{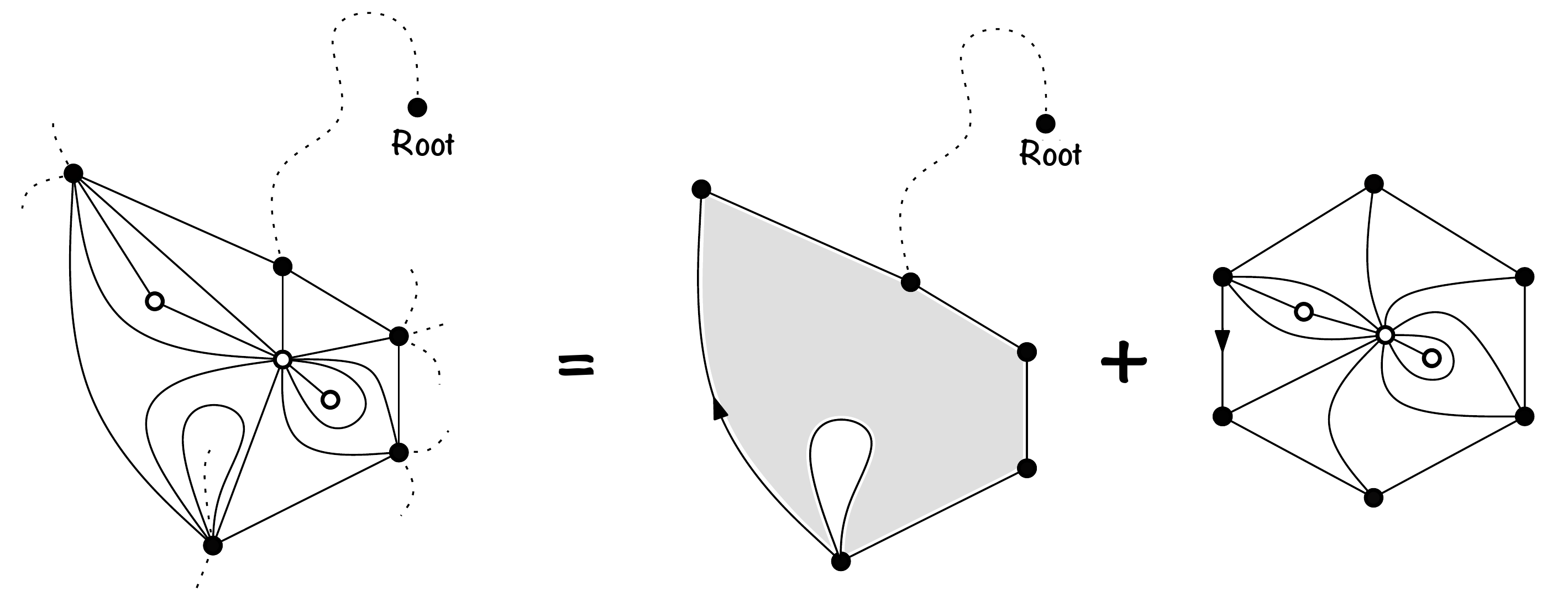}
 \caption{Isolating one site-island in a site-percolated triangulation. In the center, we have depicted in gray the face cut in the cluster $ \dot{\mathfrak{C}}$, while on the right we have depicted the site-island corresponding to this face (it is obtained by cutting along both sides of the edges of $\dot{\mathfrak{C}}$). Note that some vertices of $\dot{\mathfrak{C}}$ are duplicated, and that this leads to a simple boundary for the site island.\label{fig:cutting-site-island}}
 \end{center}
 \end{figure}

Actually, the above decomposition requires to choose a rooting convention which picks a root edge for each site-island of $ \mathfrak{t}$ and a mirror edge on the corresponding faces of the origin cluster (see Figure~\ref{fig:cutting-site-island}). However we shall not specify a precise convention since any deterministic rule (depending on $ \dot{\mathfrak{C}}$) would work for us. If $ \mathfrak{t}$ is a random critical Boltzmann triangulation, recall that the probability that $ \mathfrak{t}$ is equal to a fixed triangulation with $n$ faces is proportional to $ z_{0}^{-n}$ where $z_{0} = \sqrt[4]{432}$. If $ \mathfrak{i} $ is a site-island, and $p \in [0,1]$, we define the $p$-weight of this site-island by putting
 \begin{eqnarray} \label{def:dotw} \Ws( \mathfrak{i};p)=p^{\vopen( \mathfrak{i})}(1-p)^{\vclosed( \mathfrak{i})}z_{0}^{\fin( \mathfrak{i})}, \end{eqnarray}
where $\fin( \mathfrak{i})$ is the number of inner triangles of $ \mathfrak{i}$, and $\vopen( \mathfrak{i})$ and $\vclosed( \mathfrak{i})$ are the number of black and white \emph{inner} vertices respectively. 
We then define 
 \begin{eqnarray} \label{def:Ws} \Ws_k(p):=\sum_{ \mathfrak{i}\in \mIs_k}\Ws( \mathfrak{i};p)=\sum_{ \mathfrak{i}\in \mIs_k} p^{\vopen( \mathfrak{i})}(1-p)^{\vclosed( \mathfrak{i})}z_{0}^{\fin( \mathfrak{i})} \end{eqnarray}
where $\mIs_k$ is the set of site-islands having boundary length $k$. Now, using the above decomposition, it is clear that for any (non-atomic) planar map $ \mathfrak{c}$, the total critical Boltzmann weight of all percolated triangulations with origin cluster $ \mathfrak{c}$ is proportional to 
\begin{equation*}
 \mathbb{P}( \dot{\mathfrak{C}}(p) = \mathfrak{c}) \quad \propto \quad p^{ \vv( \mathfrak{c})} \prod_{f\in \mathsf{Face}( \mathfrak{c})} \Ws_{\deg(f)}(p).
\end{equation*}
Using the Euler formula we have $ \vv( \mathfrak{c})-2 = \sum_{ f \in \mathsf{Face}( \mathfrak{c})} (\mathrm{deg}(f)/2-1)$ and so the last equation becomes 
 \begin{eqnarray} \label{eq:weight-cluster-site} \mathbb{P}( \dot{ \mathfrak{C}(p)} = \mathfrak{c})\quad \propto \prod_{f \in \mathsf{Face}( \mathfrak{c})} p^{ \frac{ \mathrm{deg}(f)}{2}-1} \times \dot{W}_{ \mathrm{deg}(f)}(p). \end{eqnarray} 
We deduce that conditionally on the event that it is non-atomic, $ \dot{ \mathfrak{C}}(p)$ indeed follows a Boltzmann
distribution with admissible weights given by 
\begin{equation} \nonumber
\dot{q}_{k}(p) = p^{k/2-1} \dot{W}_{k}(p)
\end{equation}
for $k \geq 1$. The admissibility of this sequence (in the sense of definition \eqref{eq:1} or even \eqref{eq:2}) follows from the fact that the critical triangulation corresponds itself to the admissible weight sequence $(z_0\delta_{k,3})_{k}$.
The asymptotic form of the weights given in Theorem \ref{thm:main} in the case of site-percolation follows immediately from the asymptotic form of the sequence $\Ws_{k}(p)$ as $k \to \infty$ provided in the next proposition.

\begin{proposition}[Asymptotic weights for site-islands] \label{prop:weight-asymptotic-site} For all $p\in[0,1]$,
the total weight $\Ws_k(p)$ of the site-islands of boundary length $k$ is orthodox (as $k \to \infty$) with exponent $\dot{\beta}(p)$ (defined in Theorem \ref{thm:main}) and growth constant given by 
$$ \frac{{z}_{0}}{1-{z}_{0}^{2/3}\dot{r}(p)},$$
where the function $\dot{r}(p)$ is defined in Proposition \ref{prop:asymptotic-form}.


\end{proposition}

\subsubsection{Reef decomposition and triangulation with boundary}

\label{sec:reef->GF}
In order to prove Proposition \ref{prop:weight-asymptotic-site} (which is done in Section \ref{sec:gf-for-site}) we describe a decomposition of site-islands into two pieces, which is illustrated in Figure~\ref{fig:necklace2}. This decomposition is inspired from the work \cite{BBG12} and already used in \cite{CKperco} with a slightly different notion of rooting, and where the authors used the word \emph{necklace} instead of \emph{reef}.
However we proceed from scratch for the reader's convenience. \medskip 

 We call \emph{empty site-island} the site-island without inner vertices (a triangle). We now consider a non-empty site-island $ \mathfrak{i}$. We call \emph{reef edges} (resp. \emph{reef triangles}) of $ \mathfrak{i}$ the inner edges (resp. non-root faces) incident to an outer vertex. We call \emph{midland} edges (resp. \emph{triangles}) the non-reef inner edges (resp. triangles). Note that Condition (ii) in the definition of a site-island $ \mathfrak{i}$ implies that reef triangles are incident to two reef-edges and either an outer edge or a midland edge. 

\fig{width=.8\linewidth}{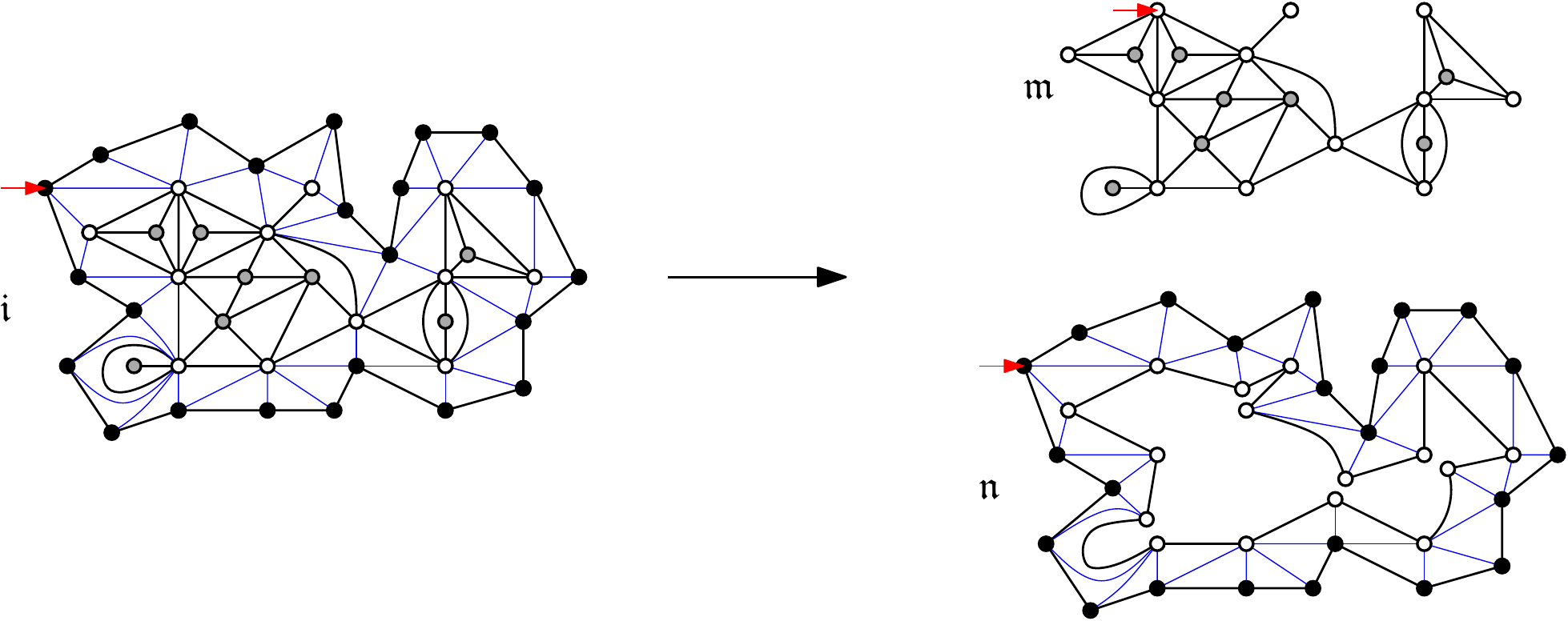}{Decomposition of a site-percolation site-island $ \mathfrak{i}$ into a midland $ \mathfrak{m}$ and a reef $ \mathfrak{n}$. In this picture, the vertices are colored either black, or white if their state is imposed, and gray otherwise. The reef edges are indicated in thin blue lines.}


We call \emph{midland} of $ \mathfrak{i}$ the map $ \mathfrak{m}$ made of the inner vertices and midland edges. It is not hard to see that $ \mathfrak{m}$ is indeed a map, that is, is connected. 
We canonically root the midland $ \mathfrak{m}$ by requesting that the reef triangle incident to the root edge of $ \mathfrak{i}$ is also incident to the root corner of $ \mathfrak{m}$. This makes $ \mathfrak{m}$ a triangulation with boundary together with a site-percolation configuration such that all outer vertices are white. Note that any triangulation with boundary can occur, including the map with a single vertex and no edge.

Now we consider the rooted map $ \mathfrak{n}$, called \emph{reef of $ \mathfrak{i}$}, obtained from $ \mathfrak{i}$ by cutting along the boundary of $ \mathfrak{m}$. If $ \mathfrak{i}$ and $ \mathfrak{m}$ have boundary length $k$ and $\ell$ respectively, then the reef $ \mathfrak{n}$ has a simple root face of degree $k$, a simple marked face of degree $\ell$, and $k+\ell$ reef triangles. More precisely, $ \mathfrak{n}$ has $k$ \emph{inward triangles} which share one edge with the root face and one vertex with the marked inner face, and $\ell$ \emph{outward triangles} which share one edge with the marked inner face and one vertex with the root face. Note that there are $$\ds {k+\ell-1 \choose k-1}$$ possible reefs of this type, since the triangle incident to the root edge is inward, and starting from there any sequence of inward and outward triangles is possible.
Moreover it is easy to see that the decomposition of non-empty site-islands into a midland and a reef is bijective. Precisely, non-empty site-islands of boundary length $k$ are in bijection with pairs $( \mathfrak{m}, \mathfrak{n})$ where 
\begin{compactitem}
\item $ \mathfrak{m}$ is a midland, that is, a non-empty triangulation with boundary together with a site-percolation configuration such that all outer vertices are white,
\item $ \mathfrak{n}$ is a reef with $k$ inward triangles and $\ell$ outward triangles, where $\ell$ is the boundary length of $ \mathfrak{m}$.
\end{compactitem}

This decomposition leads us to introduce the generating function
\begin{equation}\label{eq:gfT}
T(x,y,z):=\sum_{ \mathfrak{t}\in\mT}x^{\length( \mathfrak{t})}y^{\vout( \mathfrak{t})}z^{\ee( \mathfrak{t})},
\end{equation}
where we recall that $\mT$ is the set of triangulations with a (non necessarily simple) boundary, $\length( \mathfrak{t})$ is the boundary length of $ \mathfrak{t}$, the quantity $\vout( \mathfrak{t})$ is the number of outer vertices of $ \mathfrak{t}$, and $\ee( \mathfrak{t})$ is the number of edges. Notice that we count here triangulations according to the number of edges via the variable $z$ rather than via the number of faces as done in the preceding section. We do so because the equations we get on $T$ (see Section \ref{sec:gf}) are a bit simpler to manipulate. Using the above decomposition and summing over all percolation configurations in $ \mathfrak{t}$ we can reinterpret $\dot{W}_{k}(p)$ defined in  \eqref{def:Ws} as
 \begin{eqnarray*} \dot{W}_{k}(p) = z_{0}\times \delta_{k,3} + \sum_{\ell \geq 0}{k+\ell-1 \choose k-1} z_{0}^{\ell+k} \sum_{ \begin{subarray}{c} \mathfrak{t} \in \mathcal{T}\\ \mathrm{Length}( \mathfrak{t}) = \ell \end{subarray}} (1-p)^{\vout( \mathfrak{t})} z_{0}^{\fin( \mathfrak{t})}. \end{eqnarray*}
However, since we have $ 3 \fin( \mathfrak{t}) + \mathrm{Length}( \mathfrak{t}) = 2 \ee( \mathfrak{t})$ the last display becomes
\begin{eqnarray} \label{eq:relatingWktoF} \dot{W}_{k}(p) = z_{0} \times \delta_{k,3} + \sum_{\ell \geq 0}{k+\ell-1 \choose k-1} z_{0}^{\ell+k} z_{0}^{-\ell/3} [x^{\ell}] T(x,1-p,\tilde{z}_{0}),
\end{eqnarray}
where $[x^\ell]T$ is the coefficient of $x^\ell$ in the series $T$ and
$$\tilde{z}_{0} := (z_{0})^{2/3} = (432)^{-1/6}.$$ 
We now see that computing $\Ws_k$ boils down to counting triangulations with boundary according to the number of outer vertices (for the particular value $z= \tilde{z}_{0}$). In Section~\ref{sec:gf-for-site} we establish an algebraic equation for $T(x,y,z)$, and proceed to use generating function techniques to deduce the asymptotic behavior of $\Ws_k$ for large $k$.


\subsection{Decomposition for bond-percolation}\label{sec:weight-islands-bond}
We now consider the bond-percolation model, and define a decomposition of bond-percolated triangulations analogous to the one presented for site-percolation. Let $ \mathfrak{t}$ be a bond-percolated triangulation of the sphere. Recall our convention that the root edge must be colored in black. The origin cluster $ \overline{\mathfrak{C}}$ is the planar map obtained by keeping only those edges of the map $ \mathfrak{t}$ which are in the black cluster of the root edge (rooted at the root edge of $ \mathfrak{t}$). 
\subsubsection{Isolating the islands}
\label{sec:island-bond}
Exactly as in the site-percolation setup, we imagine that we cut along (both sides) of the edges belonging to the bond-percolation cluster of the origin $ \overline{\mathfrak{C}}$. This separates a map from each face of $ \overline{\mathfrak{C}}$, and we call these maps \emph{bond-islands}.
Let us give a precise characterization of these maps. 
\begin{definition} A \emph{bond-island} is a triangulation with simple boundary, together with a bond-percolation configuration such that 
\begin{compactitem}
\item[(i)] all the outer edges are black,
\item[(ii)] all the inner edges incident to an outer vertex are white.
\end{compactitem}
See Figure~\ref{fig:island-bond} for an example.
\end{definition}

\begin{figure}[!h]
 \begin{center}
 \includegraphics[height=4cm]{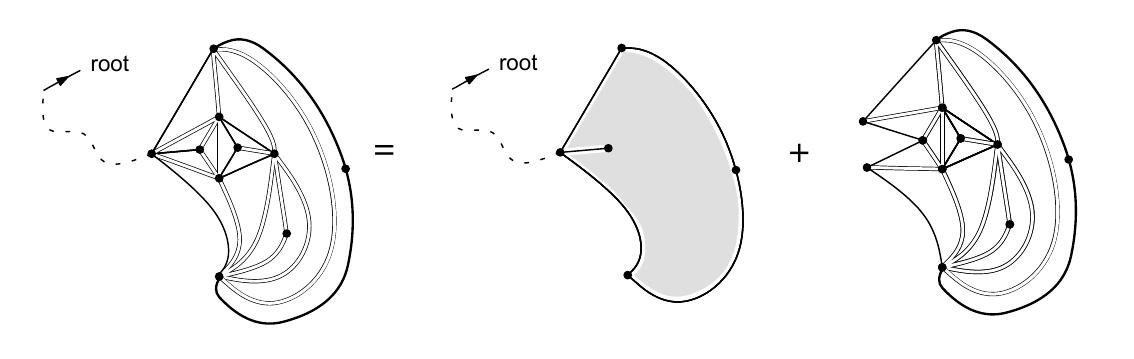}
 \caption{Isolating one bond-island in a bond-percolated triangulation. In the center, we have depicted in gray the face cut in the cluster $ \overline{\mathfrak{C}}$, while on the right we have depicted the bond-island corresponding to this face (note that it has a simple boundary).\label{fig:island-bond}}
 \end{center}
 \end{figure}

Here again, one implicitly use a rooting convention for the bond-islands and for the faces of the cluster. We then proceed as above and define the $p$-\emph{weight} of a bond-island $ \mathfrak{i}$ by
$$\Wb( \mathfrak{i};p)=p^{\eopen( \mathfrak{i})}(1-p)^{\eclosed( \mathfrak{i})}z_{0}^{\fin( \mathfrak{i})},$$ where $z_{0}=1/ \sqrt[4]{432}$ and $\eopen( \mathfrak{i})$ and $\eclosed( \mathfrak{i})$ are respectively the number of black and white \emph{inner} edges of $ \mathfrak{i}$.
We then define accordingly
 \begin{eqnarray} \label{def:Wb}\Wb_k(p):=\sum_{ \mathfrak{i}\in \mIe_k}\Wb( \mathfrak{i};p)=\sum_{ \mathfrak{i}\in \mIe_k} p^{\eopen( \mathfrak{i})}(1-p)^{\eclosed( \mathfrak{i})}z_{0}^{\fin( \mathfrak{i})}, \end{eqnarray}
where $\mIe_k$ is the set of bond-islands having boundary length $k$. By the above decomposition, we can then compute the probability that a critical Boltzmann bond-percolated triangulation has the origin cluster equal to a fixed $ \mathfrak{c}$, and the later is proportional to 
\begin{equation}\label{weight-cluster-bond}
 \mathbb{P}( \overline{\mathfrak{C}}(p) = \mathfrak{c}) \propto p^{ \ee( \mathfrak{c})} \prod_{f\in \mathsf{Face}( \mathfrak{c})} \Wb_{\deg(f)}(p),
\end{equation}
Using the fact that $ \ee( \mathfrak{c}) = \sum_{ f \in \mathsf{Face}( \mathfrak{c})} \mathrm{deg}(f)/2$ we can then reinterpret the last display in a similar fashion as \eqref{eq:weight-cluster-site}. We conclude that $ \overline{ \mathfrak{C}}$ indeed follows an admissible Boltzmann distribution with weight sequence given by $ \overline{q}_{k}(p) = p^{k/2} \overline{W}_{k}(p)$. 
Again, the admissibility of this sequence (in the sense of definition \eqref{eq:1} or even \eqref{eq:2}) follows from the fact that the critical triangulation corresponds itself to the admissible weight sequence $(z_0\delta_{k,3})_{k}$.
Also, the asymptotic form of the weight sequence $ \overline{\mathbf{q}}(p)$ given in Theorem \ref{thm:main} is a direct consequence of the following proposition: 
\begin{proposition}[Asymptotic weights for bond-islands] \label{prop:weight-asymptotic-bond} For all $p\in [0,1]$ the total weight $\Wb_k(p)$ of the bond-islands of boundary length $k$ is orthodox (as $k \to \infty$) with exponent $\overline{\beta}(p)$ (defined in Theorem \ref{thm:main}) and growth constant given by 
$$ \frac{\overline{r}(p)}{(1-p) z_{0}^{1/3}},$$
where $ \overline{r}(p)$ is defined in Proposition \ref{prop:asymptotic-form-bond}.

\end{proposition}

\subsubsection{Generating function reduction}
Similarly as in Section \ref{sec:reef->GF} we present here the generating function that we will use in order to prove Proposition \ref{prop:weight-asymptotic-bond}. It should be clear from the above definition of bond-islands that computing the weight $\Wb_k$ boils down to counting triangulations with a simple boundary of length $k$ according to the number of \emph{reef edges} (inner edges incident to an outer vertex). Formally, we denote by $\mS'$ the set of triangulations with a simple boundary together with the map made of one edge and two vertices, and we denote
\begin{equation}\label{eq:gfS}
S(x,y,z)=\sum_{ \mathfrak{t}\in\mS'}x^{\length( \mathfrak{t})}y^{\reef( \mathfrak{t})}z^{\ee( \mathfrak{t})},
\end{equation}
where $\length( \mathfrak{t})$ is the boundary length of $ \mathfrak{t}$, and $\reef( \mathfrak{t})$ is the number of edges incident to an outer vertex. Notice that denoting $\ereef( \mathfrak{t})$ the number of reef edges (that is, inner edges incident to an outer vertex),  we have $\reef( \mathfrak{t})=\ereef( \mathfrak{t})+\length( \mathfrak{t})$. Moreover, denoting $\fin( \mathfrak{t})$ the number of internal triangles of $\mathfrak{t}$,   we have $ 3\fin( \mathfrak{t}) + \mathrm{Length}( \mathfrak{t}) = 2 \ee ( \mathfrak{t})$. Using these relations, and recalling the notation  $ \tilde{z}_{0} = z_{0}^{2/3} = (432)^{-1/6}$ we obtain 
\begin{equation}\label{eq:weight-as-coef-bond}
\Wb_k(p)=\frac{1}{(1-p)^k z_{0}^{k/3}}\,[x^k]\big(S(x,1-p,\tilde{z}_{0})-x^2(1-p) \tilde{z}_{0}\big),
\end{equation}
because the map with one edge and two vertices contributes $x^2(1-p)z$ to $S(x,1-p,z)$. In Section~\ref{sec:gf-for-site} we establish an algebraic equation for $S(x,y,z)$, and proceed to use generating function techniques to deduce the asymptotic behavior of $\Wb_k$ for large $k$. In particular Proposition \ref{prop:weight-asymptotic-bond} is a direct consequence of \eqref{eq:weight-as-coef-bond} together with the forthcoming Proposition \ref{prop:asymptotic-form-bond}.



\section[Proofs by generating functionology]{Weight of islands via a generating function approach}\label{sec:gf}
In this section we prove Propositions \ref{prop:weight-asymptotic-site} and \ref{prop:weight-asymptotic-bond}. As we have already noticed, for the site-percolation model, this boils down to the enumeration of triangulations with boundary according to the number of outer vertices. For the bond-percolation model, this boils down to the enumeration of triangulations with simple boundary according to the number of edges incident to outer vertices.

\subsection{Site-percolation case}\label{sec:gf-for-site}
\subsubsection{Triangulations with boundary and outer vertices}
Recall that the generating function $T(x,y,z)$ is defined by~\eqref{eq:gfT}.
\begin{lemma}
The generating function $T(x)\equiv T(x,y,z)$ satisfies the following functional equation:
\begin{equation}\label{eq:funF}
T(x)=y+x^2z\,T(x)^2+\frac{z\,(y-1)(T(x)-y)^2}{y\,x\,T(x)}+\frac{z}{yx}(T(x)-y-x\,T_1),
\end{equation}
where $T_1=[x]T(x)$. 
\end{lemma}

\begin{proof}
This result translates a recursive decomposition of maps (\emph{\`a la Tutte}). We first partition the set $\mT$ according to the situation around the root: a map $ \mathfrak{t}$ in $\mT$ is either 
\begin{compactitem}
\item[(i)] the atomic map,
\item[(ii)] or a non-atomic map such that the root edge is a bridge,
\item[(iii)] or a non-atomic map such that the root edge is not a bridge.
\end{compactitem}
\fig{width=.7\linewidth}{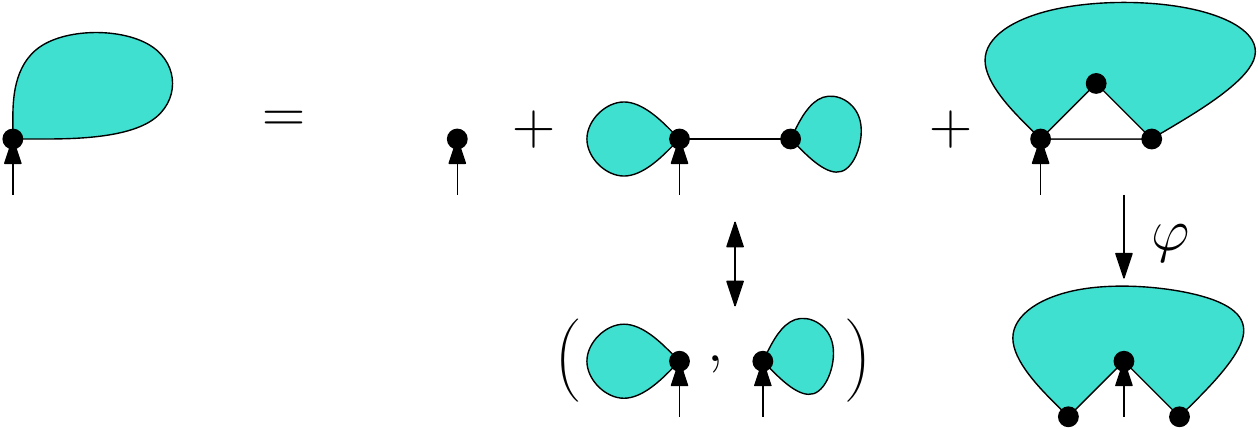}{Decomposition of triangulations with boundary, by deletion of the root edge. The arrows indicate the root-corners.}
This situation is represented in Figure~\ref{fig:decomposition-island-site}. The atomic map contributes $y$ to $T(x)$. For a map $ \mathfrak{t}$ satisfying (ii), deleting the root edge gives a pair of maps in $\mT$. This gives a bijection between maps corresponding to case (ii) and pairs of maps in $\mT$, and shows that these maps contribute $\ds x^2z\,T(x)^2$ to $T(x)$ (where the factor $x^2z$ accounts for the bridge). Finally, for a map $\mathfrak{t}$ satisfying (iii), we consider the inner triangle $t$ incident to the root edge, and define $\varphi(\mathfrak{t})$ as the map obtained by deleting the root edge and transferring the root corner to the corner which was formerly opposite to the root edge in the triangle $t$; see Figure~\ref{fig:decomposition-island-site}. The mapping $\varphi$ is a bijection between maps corresponding to case (iii) and triangulations with boundary of length at least 2. Moreover $\vout(\mathfrak{t})=\vout(\varphi(\mathfrak{t}))$ if the root vertex of $\varphi(\mathfrak{t})$ is a \emph{cut-point} (i.e. deleting it disconnects the map), and $\vout(\mathfrak{t})=\vout(\varphi(\mathfrak{t}))-1$ otherwise; see Figure~\ref{fig:decomposition-island-site2}(a). We define 
$$\hT(x)\equiv \hT(x,y,z)=\sum_{\mathfrak{t}\in\hmT}x^{\length(\mathfrak{t})}y^{\vout(\mathfrak{t})}z^{\ee(\mathfrak{t})},$$
where $\hmT$ is the set of of triangulations with boundary such that the root vertex is not a cut-point.
It is easy to see that the maps such that the root vertex of $\varphi(\mathfrak{t})$ is a cut-point contribute $\ds \frac{z}{x}(T(x)-\hT(x))$ to $T(x)$, while the maps such that the root vertex of $\varphi(\mathfrak{t})$ is a not cut-point contribute $\ds \frac{z}{yx}(\hT(x)-y-xT_1)$. 
Thus adding the contributions of cases (i-iii) gives 
\begin{equation}\label{eq:decompose-triang}
T(x)=y+ x^2z\,T(x)^2+\frac{z}{x}(T(x)-\hT(x))+\frac{z}{yx}(\hT(x)-y-x\,T_1).
\end{equation}
Lastly, we observe that the non-atomic maps in $\mT$ are in bijection with non-empty sequences of non-atomic maps in $\hmT$; see Figure~\ref{fig:decomposition-island-site2}(b). This gives 
$$T(x)-y=\frac{\hT(x)-y}{1-(\hT(x)-y)/y}.$$
Solving for $\hT(x)$ and plugging the result in~\eqref{eq:decompose-triang} gives~\eqref{eq:funF}.
\end{proof}

\fig{width=.8\linewidth}{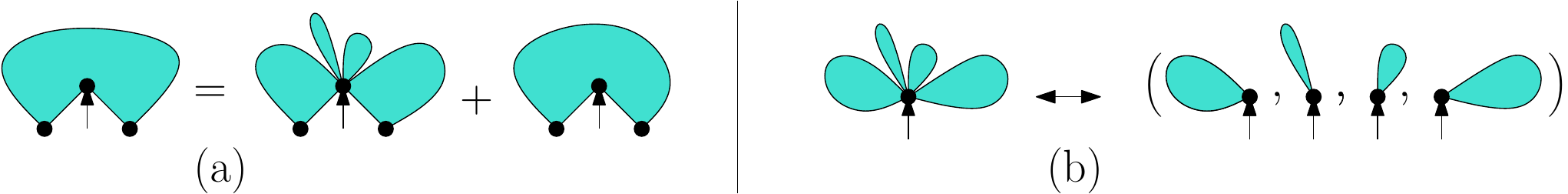}{(a) Partition of the set of triangulations of boundary length at least 2. (b) Decomposition of a non-atomic map in $\mT$ into a sequence of non-atomic maps in $\hmT$.}

We then specialize the series to the value $z= \tilde{z}_{0}= z_{0}^{2/3}=(432)^{-1/6}$ and introduce:
 \begin{eqnarray} \mathbf{T}(x,y) &:=& T(x,y,\tilde{z}_{0}) = \sum_{ \mathfrak{t}\in\mT}x^{\length(\mathfrak{t})}y^{\vout(\mathfrak{t})}\tilde{z}_{0}^{\ee(\mathfrak{t})}. \end{eqnarray}
\begin{lemma}\label{lem:algebraic-site}
The series $\mathbf{T}$ satisfies the following algebraic equation:
\begin{equation}\label{eqH}
{\mathbf{T}}^{3}{x}^{3}-{\mathbf{T}}^{2}x{2}^{2/3}\sqrt {3}+\mathbf{T}xy{2}^{2/3}\sqrt{3}
-\mathbf{T}x{2}^{-1/3}\sqrt{3}+3\,\mathbf{T}x{2}^{-4/3}+{\mathbf{T}}^{2}-2\,\mathbf{T}y+\mathbf{T}+{y}^{2}-y=0.
\end{equation}
\end{lemma}

\begin{proof}
We first establish an equation for $\tT(x)=T(x,1,z)$; the computations can be found in the Maple session accompanying this paper. Setting $y=1$ in~\eqref{eq:funF} gives 
$$\tT(x)=1+x^2z\,\tT(x)^2+\frac{z}{x}(\tT(x)-1-x\,\tT_1),$$
where $\tT(x)=T(x,1,z)$ and $\tT_1=[x]\tT(x)$.
From this we can obtain an algebraic equation for $\tT_1$ by applying the standard \emph{quadratic method} (see \cite{Goulden:Lagrange}). More precisely this equation (already obtained by Tutte \cite{Tutte:census3}) reads $Q(\tT_1,z)=0$, where 
\begin{equation}\label{eq:polT1} 
Q(u,z)=64\,{u}^{3}{z}^{5}-96\,{u}^{2}{z}^{4}-27\,{z}^{5}+30\,u{z}^{3}+{u}^{2}z+{z}^{2}-u.
\end{equation}
Next, we observe that $T_1=y\tT_1$. Eliminating $T_1$ between the equation $Q(T_1/y,z)=0$ and~\eqref{eq:funF} (e.g. using resultant) gives an equation of the form 
\begin{equation}\label{eq:T}
R(T(x,y,z),x,y,z)=0
\end{equation}
where $R(u,x,y,z)$ is a polynomial of degree 9 in $u$ (see Maple session). At $z=\tilde{z}_{0}=(432)^{-1/6}$ this polynomial factorizes as 
$$R(u,x,y,\tilde{z}_{0})={R_1(u,x,y)}^2R_2(u,x,y),$$
where 
$$
R_1(u,x,y)={u}^{3}{x}^{3}-{u}^{2}x{2}^{2/3}\sqrt {3}+uxy{2}^{2/3}\sqrt{3}
-ux{2}^{-1/3}\sqrt {3}+3\,ux{2}^{-4/3}+{u}^{2}-2\,uy+u+{y}^{2}-y,
$$
and $R_2(u,x,y)=R_{1}(u,x,y)-(27\times 2^{2/3})\,x\,\mathbf{T}/16$.
Hence, either $R_1(\mathbf{T},x,y)=0$, or $R_2(\mathbf{T},x,y)=0$. We know that $\mathbf{T}(x,y)$ is a series in $\Q[y][[x]]$ such that $[x^0]\mathbf{T}=y$. The only series $S=S(x,y)$ with these properties satisfying $R_2(S,x,y)=0$ has some negative coefficients (e.g. the coefficient of $x^2y^2$ is negative), hence is distinct from $\mathbf{T}(x,y)$. Thus we conclude that $R_1(\mathbf{T},x,y)=0$, which is precisely~\eqref{eqH}.
\end{proof}

\begin{proposition}\label{prop:asymptotic-form}
For all $y\in[0,1]$, the coefficients $[x^n]\mathbf{T}(x,1-y)$ of $\mathbf{T}(x,1-y)=T(x,1-y,\tilde{z}_{0})$ are orthodox sequences in $n\geq 1$ with growth constant $\dot{r}(y)$ and exponent $\dot{\beta}(y)$ where $\dot{\beta}$ is defined Theorem \ref{thm:main} and for $y\in (\dot{p}_{c},1]$ the growth constant $\dot{r}(y)$ is a root of 
$$8\,x^3-12\sqrt{3}\,{2}^{2/3}\,x^2+9\,\sqrt[3]{2}\,x+15\sqrt{3}+54(1-y)-27,$$
while for $y\in [0, \dot{p}_{c})$ we have $\ds \dot{r}(y)=2^{-1/3}\,(\sqrt{3}+1-2y).$ 
\end{proposition}

\begin{rek} Notice that we used the variable $1-y$ instead of simply $y$ in the above result in order to make the connection with the above Proposition \ref{prop:weight-asymptotic-site} and Theorem \ref{thm:main} clearer.
\end{rek}

\begin{rek}[Links with \cite{CKperco}] \label{rek:CK} The above proposition in the case $ y \leq \dot{p}_{c}$ could directly be deduced from \cite{CKperco} and probably also for $y > \dot{p}_{c}$ with a little more work. Notice in particular that the weights $[x^{n}] \mathbf{T}(x,y)$ can be related to $Q_{a}(\{ \mbox{triangulations with a boundary of length }n\})$ where the measure $Q_{a}$ is defined in \cite[Eq. (12)]{CKperco} with $a = y$ and estimated in \cite[Proposition 3.2]{CKperco}. Specifically with the notation of \cite[Proposition 3.2]{CKperco} we have with $a=y$ 
$$ 2^{-1/3}\big(\sqrt{3}+1-2(1-y)\big) = \left( \frac{r_{c}(2a+\gamma)}{2}\right) \frac{12}{\tilde{z}_{0}}.$$ The exponent $5/2$ or $5/3$ of the above proposition is in this framework related to exponent of the tail of the probability that a certain (subcritical or critical) Galton--Watson tree has $n$ vertices, see \cite[Eq. (17)]{CKperco} in the case $1-y= \dot{p}_{c}= 1/2$. 
However we proceed in this paper with a totally different purely analytic approach compared to the probabilistic one in \cite{CKperco}.
\end{rek}

\begin{proof}
 For any given value $y\geq 0$, the algebraic equation~\eqref{eqH} gives an equation for $\mathbf{T}(x)\equiv \mathbf{T}(x,y)$ of the form $P(\mathbf{T}(x),x)=0$, where $P(u,x)$ is a polynomial in $\mathbb{R}[u,x]$. Since $\mathbf{T}(x)$ is a power series with non-negative coefficients, one can use standard methods of analytic combinatorics (see \cite{Flajolet:analytic}) in order to obtain the asymptotic form of the coefficients $[x^n]\mathbf{T}(x)$. In particular, we know (see \cite[Chapter VII.7]{Flajolet:analytic}) that, unless $\mathbf{T}(x)$ has several dominant singularities, its coefficients have an asymptotic behavior of the form
\begin{equation}\label{eq:asymptotic-form}
[x^n]\mathbf{T}(x)\equiv c\,n^{\alpha}\rho^{-n},
\end{equation}
where $\rho>0$ is the radius of convergence of $\mathbf{T}(x)$, and $\al$ is a rational number.
Moreover, $\rho$ is either a root of the leading coefficient $C(x)$ of $P$ with respect to $u$, or a root of the discriminant $\Delta(x)$ of $P(u,x)$ with respect to $u$. Here $C(x)=4x^3$ so $\rho$ is a root of $\Delta(x)$. Moreover, the exponent $\al$ is determined by the type of singularity of $\mathbf{T}(x)$ at $x=\rho$. The difficulty here is to make this analysis uniform in~$y$. Below we first explain in some details the cases $y=1$ and $y=1/2$, and then sketch the proof for the other values $y\in[0,1]$. The accompanying Maple session contains all the necessary computations.

Let us first set $y=1$. In this case,~\eqref{eqH} gives $P(\mathbf{T}(x),x)=0$, where 
$$P(u,x)=u\,\left(4\,x^3u^{2}+ \left( -4\,x{2}^{2/3}\sqrt {3}+4 \right) \,u+2\,x{2}^{2/3}\sqrt {3}+3\,x{2}^{2/3}-4\right).$$
The discriminant of $P(u,x)/u$ is 
$$\Delta(x)={2}^{2/3}/3 \left( 2\,\sqrt {3}+3 \right) \left( 6\,x+3\,\sqrt [3]
{2}+\sqrt {3}\sqrt [3]{2} \right) \left( \sqrt {3}\sqrt [3]{2}-\sqrt 
[3]{2}-2\,x \right) ^{3}.$$
The only positive root of $\Delta(x)$ is $x=2^{-2/3}(\sqrt{3}-1)$, so $\rho=2^{-2/3}(\sqrt{3}-1)$. There is no other dominant singularities for $\mathbf{T}(x)$ (as these would need to be other roots of $\Delta(x)$ of modulus $\rho$), so~\eqref{eq:asymptotic-form} holds, and it remains to determine the singular behavior of $\mathbf{T}(x)$ around $x=\rho$. First, we determine $\mathbf{T}(\rho)$.
Since we know that $\mathbf{T}(x)$ is singular at $x=\rho$, we conclude that $\mathbf{T}(\rho)$ is the double root of the polynomial $Q(u)=P(u,\rho)$. This gives $\mathbf{T}(\rho)=(\sqrt{3}+1)/2$. Then, the singular behavior of $\mathbf{T}(x)$ at $x=\rho$ is determined from the expansion of the curve $P(u,x)=0$ around $(u,x)=(\mathbf{T}(\rho),\rho)$. This expansion, in turn, can be determined using Newton's polygon method (which is implemented in the \emph{Puiseux} command of Maple). This gives
$$\mathbf{T}(x)=_{x\to\rho} \mathbf{T}(\rho)-{2}^{-1/3} \left( 2\,\sqrt {3}+3 \right) (\rho-x)+\left( 7\,\sqrt {2}+4\,\sqrt {6} \right)\,(\rho-x)^{3/2}+o((\rho-x)^{3/2}).$$
The singular part is of order $(\rho-x)^{3/2}$, therefore $\al=-1-3/2=-5/2$. Thus for $y=1$, 
$$[x^n]\mathbf{T}(x)\equiv c\,n^{-5/2}\,\left(2^{-2/3}(\sqrt{3}-1)\right)^{-n},$$
for some constant $c$ (which could also be determined from the above). This indeed gives $ \dot{\beta}(0) = 5/2$ and $\dot{r}(0) = (2^{-2/3}( \sqrt{3}-1	))^{-1} = 2^{-1/3}(\sqrt{3}+1)$. 

Next, we treat the case $y=1/2$. In this case, the discriminant $\Delta(x)$ has two positive roots $x=2^{4/3}/(5\sqrt{3})$ and $x=2^{1/3}/\sqrt{3}$. Moreover the radius of convergence $\rho$ of $\mathbf{T}(x)$ at $y=1/2$ needs to be larger than or equal to the radius $2^{-2/3}(\sqrt{3}-1)$ obtained for $y=1$ (since the coefficients of $\mathbf{T}(x)$ are increasing in $y$). Thus $\rho=2^{1/3}/\sqrt{3}$. Proceeding as above we find 
 $$\mathbf{T}(x)=_{x\to\rho} \mathbf{T}(\rho)-{2}^{-11/9}{3}^{7/6}\,(\rho-x)^{2/3}+o((\rho-x)^{2/3}).$$
 The singular part is of order $(\rho-x)^{2/3}$, hence for $y=1/2$, 
$$[x^n]\mathbf{T}(x)\equiv c\,n^{-5/3}\,(2^{1/3}/\sqrt{3})^{-n},$$
for some constant $c$.

Now let us consider a generic value of $y$ in $[0,1]$. 
Note that for $y=0$, $\mathbf{T}(x)=0$ so we now suppose $y>0$.
The discriminant $\Delta(x)$ of $P(u,x)$ factorizes as $\Delta(x)=16 \Delta_1(x)\Delta_2(x)^3$ where
\begin{eqnarray}
\Delta_1(x)&=& (27-15\sqrt{3}-54y)\,x^3-9\,\sqrt[3]{2}\,{x}^{2}+12\sqrt{3}\,{2}^{2/3}\,x-8, \label{eq:Delta1}\\
\Delta_2(x)&=& (\sqrt{3}+2y-1)\,x-2^{1/3}. \label{eq:Delta2}
\end{eqnarray}
Let us denote by $\tau_1(y),\tau_2(y),\tau_3(y)$ the three roots of $\Delta_1(x)$, and by $\sigma(y)=2^{1/3}/(\sqrt{3}+2y-1)$ the root of $\Delta_2(x)$. We know that the radius of convergence $\rho(y)$ of $\mathbf{T}(x)$ is in $\{\tau_1(y),\tau_2(y),\tau_3(y),\sigma(y)\}$. These are all real numbers for $y$ in $[0,1]$ and all distinct for $y\in [0,1]\setminus\{1/2\}$ (since the discriminant of the polynomial $\Delta_1(x)\Delta_2(x)$ is non-zero). 
Moreover $\{|\tau_1(y)|,|\tau_2(y)|,|\tau_3(y)|,|\sigma(y)|\}$ are also distinct (since the resultant of the polynomials $\Delta(x)$ and $\Delta(-x)$ is non-zero), hence $\mathbf{T}(x)$ does not have several dominant singularities and the asymptotic form~\eqref{eq:asymptotic-form} is valid.

We will see that $\rho(y)=\sigma(y)$ if and only if $y\in[1/2,1]$. First observe that from the above $\rho(1)=\sigma(1)$. Moreover $\sigma(y)>\max\{\tau_1(y),\tau_2(y),\tau_3(y)\}$ for all $y\in(1/2,1]$ (since this is true for $y=1$). Now since $\rho(y)$ is decreasing in $y$, these two facts imply $\rho(y)=\sigma(y)$ for all $y\in(1/2,1]$. We can also determine the value of $\mathbf{T}(\sigma(y))$, since it has to be the double root of the polynomial $P(u,\sigma(y))$. Then applying Newton's polygons method, we get the following expansion 
$$\mathbf{T}(x)\underset{x\to \sigma(y)}{=}\mathbf{T}(\sigma(y))
-\frac{\sqrt {3} (\sqrt{3}+2y-1)^{2}}{2^{4/3}(2y-1)}(\sigma(y)-x)
+\frac{(\sqrt {3}+2\,y-1)^4\sqrt{2y-1}}{2\sqrt{2}(2y-1)^3}(\sigma(y)-x)^{3/2}
+o((\sigma(y)-x)^{3/2}),
$$
for all $y\in(1/2,1]$. The singular part is of order $(\sigma(y)-x)^{3/2}$, hence for all $y\in (1/2,1]$, 
$$[x^n]\mathbf{T}(x)\equiv c(y)\,n^{-5/2}\sigma(y)^{-n},$$
for some constant $c(y)$.
Moreover, since the coefficient of $(\sigma(y)-x)^{3/2}$ is imaginary for $y<1/2$, we have $\rho(y)\neq \sigma(y)$ for $y\in [0,1/2)$. Thus for $y\in[0,1/2)$, $\rho(y)\in \{\tau_1(y),\tau_2(y),\tau_3(y)\}$ is a root of $\Delta_1(x)$. This implies that the term $\dot{r}(1-y)$ in the expansion~\eqref{eq:asymptotic-form} is a root of $x^3\Delta_1(1/x)$ as stated in the proposition.

It remains to determine the singular behavior of $\mathbf{T}(x)$ at $\rho(y)\in \{\tau_1(y),\tau_2(y),\tau_3(y)\}$ for $y$ in $ [0,1/2)$. One could try to substitute $x=\tau_i(y)$ in the equation $P(\mathbf{T}(x),x)=0$, and proceed as above. Unfortunately, the expressions for the roots $\tau_i$ are rather complicated, and we failed to get Maple to determine $\mathbf{T}(\rho(y))$ and the expansion at $(\rho(y),\mathbf{T}(\rho(y)))$. Instead, we computed (using polynomial eliminations) a polynomial $Q(U,X)$ (whose coefficients depend on $y$) such that
$$Q(\mathbf{T}(\rho(y))-\mathbf{T}(x),\rho(y)-x)=0,$$
for all $y$ in $[0,1/2)$. The polynomial $Q(U,X)$ is obtained as follows. 
We first define $A(u)=\res(P(u,x),\Delta_1(x),x)$, so that $A(\mathbf{T}(\rho(y)))=0$ (because $x=\rho(y)$ is root of $P(\mathbf{T}(\rho(y)),x)$ as well of $\Delta_1(x)$). 
Define $B(U,x)=\res(A(u),P(u-U,x),u)$, so that $B(\mathbf{T}(\rho(y))\!-\!\mathbf{T}(x),\,x)=0$ (because $u=\mathbf{T}(\rho(y))$ is root of both $A(u)$ and $P(u-(\mathbf{T}(\rho(y))\!-\!\mathbf{T}(x)),\,x)$). 
Lastly, define $Q(U,X)=\res(\Delta_1(x),B(U,x-X),x)$, so that $Q(\mathbf{T}(\rho(y))\!-\!\mathbf{T}(x),\,\rho(y)\!-\!x)=0$ (because $v=\rho(y)$ is root of both $\Delta_1(v)$ and $B(\mathbf{T}(\rho(y))\!-\!\mathbf{T}(x),\,v-(\rho(y)\!-\!x))$).

By examining the curve $Q(U,V)=0$ around $(U,V)=(0,0)$, one can determine the singular behavior of $\mathbf{T}(x)$ around $\rho(y)$ by Newton's polygon method. This leads to the expansion 
$[x^n]\mathbf{T}(x,y)\sim_{n \to \infty} c(y)\,n^{-3/2}\,\dot{r}(1-y)^n$
stated in Proposition~\ref{prop:asymptotic-form}. 
An extra complication actually appears in Newton's polygon method for some particular values $y_1,y_2,y_3$ of $y$, because some coefficients of the polynomial $Q(U,V)$ become 0. The values $y_1,y_2,y_3$ are the roots of 
$R(y)=(1944\,\sqrt {3}{y}^{2}+5832\,{y}^{3}-1944\,\sqrt {3}y-8748\,{y}^{2}+
252\,\sqrt {3}+3794\,y-439)(5\sqrt{3}+18y-9)$
lying in the interval $[0,1/2)$. These values can be treated separately (using again polynomial elimination), and lead to the same expansion as the generic values of $y$ in $[0,1/2)$ (all computations are available in the Maple session). 
\end{proof}

We will now deduce from Proposition~\ref{prop:asymptotic-form} the asymptotic behavior of the weights $\Ws_n(p)$ stated in Proposition~\ref{prop:weight-asymptotic-site}.
\begin{proof}[Proof of Proposition~\ref{prop:weight-asymptotic-site}]
We recall \eqref{eq:relatingWktoF} assuming $n > 3$:
$$ \dot{W}_{n}(p) = z_{0}^{n} \sum_{\ell \geq 0}{n+\ell-1 \choose n-1} \tilde{z_{0}}^{\ell} [x^{\ell}] \mathbf{T}(x,1-p).$$
From Proposition~\ref{prop:asymptotic-form} we know the asymptotic for $[x^{\ell}] \mathbf{T}(x,1-p)$; we will now use a residue calculation to show that this implies Proposition \ref{prop:weight-asymptotic-site}. First, we write the last display as

\begin{eqnarray*}
\Ws_n(p)&=&z_{0}^n[x^0]\left(\sum_{i=0}^\infty [x^{i}] \mathbf{T}(x,1-p)\tilde{z}_{0}^ix^i \right)\left(\sum_{j=0}^\infty {n+j-1\choose j}x^{-j}\right)\\
&=&{z}_{0}^n[x^0] F(x)(1-1/x)^{-n},
\end{eqnarray*}
where $F(x)\equiv F(x,p)= \mathbf{T}(\tilde{z}_{0}x,\,1-p)$. We know that $F(x)$ is algebraic, and has unique dominant singularity $\rho=\rho(p)=\frac{1}{\tilde{z}_{0} \dot{r}(p)}$, where $\dot{r}(p)$ is the growth constant defined in Proposition~\ref{prop:asymptotic-form}. Hence $F(x)$ admits an analytic continuation in a domain $\Omega$ of the form $\{x\in\mathbb{C},~|x|<\theta\}\setminus [\rho,\infty)$ with $\theta>\rho$. We now fix $\rho'\in (\rho,\theta)$, and define $\gamma_n$ as the curve represented in Figure~\ref{fig:Hankel-contour}. We perform the contour integral along $\gamma_{n}$.
\fig{width=4cm}{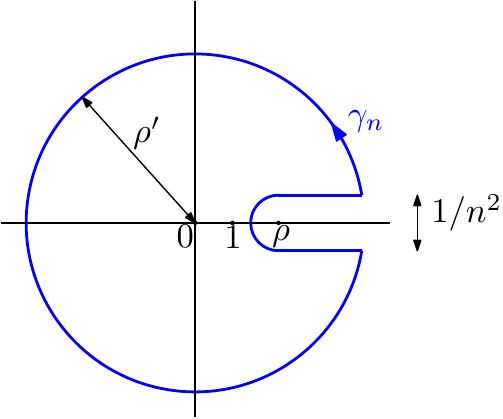}{The Hankel contour used in the residue calculation.}

We note that $\dot{r}(p)<1/\tilde{z}_{0}$ for all $p$ (since $\dot{r}(p)\leq \dot{r}(0)=2^{-1/3}(\sqrt{3}+1)< 1/ \tilde{z}_{0}$), hence $\rho>1$. So the factor $\left(1-\frac{1}{x}\right)$ is largest for $x$ in the part of $\gamma_n$ close to $\rho$. In particular, the part of $\gamma$ on the circle $|x|=\rho'$ has asymptotically negligible contribution. More precisely, making the change of variable $x=\rho\,(1+u/n)$, we get
 \begin{eqnarray*}
\Ws_n(p)&=&{z}_{0}^n\oint_{\gamma_n} F(\rho\,(1+u/n))\left(1-\frac{1}{\rho\,(1+u/n)}\right)^{-n}\frac{du}{2i\pi (n+u)},\\
&\underset{n \to \infty}{\sim}& \frac{{z}_{0}^n}{n}\int_0^{n\rho'} \left(F(\rho(1+t/n)_+)-F(\rho(1+t/n)_-)\right)\left(1-\frac{1}{\rho\,(1+t/n)}\right)^{-n}\frac{dt}{2i\pi},\\
&\underset{n \to \infty}{\sim}& \frac{{z}_{0}^n}{n}\left(1-\frac{1}{\rho}\right)^{-n}\int_0^{\infty} \left(F(\rho(1+t/n)_+)-F(\rho(1+t/n)_-)\right)e^{-t/(\rho-1)}\frac{dt}{2i\pi},
 \end{eqnarray*}
where for $x\in [\rho,\sigma]$, $F(x_+)$ and $F(x_-)$ denote the value of $F$ above and under the cut respectively. Now our asymptotic (Proposition \ref{prop:asymptotic-form}) implies that the singularity of $F(x)$ at $x=\rho$ has the form $\tc'(\rho-x)^{\dot{\beta}(p)-1}$, so that 
\begin{eqnarray*}
\left(F(\rho(1+t/n)_+)-F(\rho(1+t/n)_-)\right)&\underset{n \to \infty}{\sim}& \tc'\left( (-\rho t/n)^{\dot{\beta}(p)-1}_+(-\rho t/n)^{\dot{\beta}(p)-1}_-\right)\\
&\underset{n \to \infty}{\sim}& \tc'(\rho t/n)^{\dot{\beta}(p)-1}2i\sin(\dot{\beta}(p)\pi).
\end{eqnarray*}
Thus, 
 \begin{equation}
\Ws_n(p)\underset{n \to \infty}{\sim}\left(\frac{{z}_{0}}{1-\tilde{z}_{0}\dot{r}(p)}\right)^{n}n^{-\dot{\beta}(p)}\frac{\tc' \sin(\dot{\beta}(p)\pi)}{\pi}\int_0^{\infty} (\rho t)^{\dot{\beta}(p)-1} e^{-t/(\rho-1)}dt.
\end{equation}
This completes the proof of Proposition \ref{prop:weight-asymptotic-site}.
\end{proof}

\subsubsection{Analysis of the derivative}

In Section \ref{sec:size} we will derive information about the size (i.e.~number of edges) of the hulls of critical site-percolation clusters. These results are based on the analysis of the derivative of the function $T$ in the variable $z$, more precisely we consider
$$  \mathbb{T}(x,y) := \left.\frac{\partial}{\partial z} T(x,y,z)\right|_{z=\tilde{z}_{0}}.$$
\begin{proposition} \label{prop:size-site} For any $y \in [0,1]$, the coefficients of $[x^{n}] \mathbb{T}(x,1-y)$ form an orthodox sequence with the same growth constant $\dot{r}(y)$ as in Proposition~\ref{prop:asymptotic-form} and exponent $\dot{\gamma}(y)$ equal to $1/2$ except in the critical case where $\dot{\gamma}(\dot{p}_{c}) = 1/3$. 
\end{proposition}

\begin{proof}
The proof of Proposition \ref{prop:size-site} follows the same strategy as that of Proposition \ref{prop:asymptotic-form} (getting an algebraic equation on $ \mathbb{T}$ and proceeding to singularity analysis). Out of concern for conciseness, we do not provide the details here; but the complete derivation is provided in the associated Maple session. However, the Maple derivation does use a preliminary claim: \emph{for all $y \in [0,1]$, the radius of convergence of the series $\mathbb{T}(x,y)$ and $\mathbf{T}(x,y)$ (considered as power series in $x$) are equal}. This claim, in turns, is an easy consequence of Lemma \ref{lem:bound-nb-edges-Sk}. 
\end{proof}

\begin{rek} Proposition \ref{prop:size-site} could also be deduced from the results in \cite{CKperco} at least in the case $y\leq \dot{p}_c$.
\end{rek}

\subsection{Bond-percolation case}
\label{sec:gf-for-bond}
\subsubsection{Triangulations with simple boundary and reef edges}
Recall the generating function $S(x,y,z)$ defined by~\eqref{eq:gfS}. As in the last section, we specialize it for $z= \tilde{z}_{0}$ and introduce:
 \begin{eqnarray} \mathbf{S}(x,y) &:=& S(x,y, \tilde{z}_{0}) = \sum_{ \mathfrak{t}\in\mS'}x^{\length( \mathfrak{t})}y^{\reef( \mathfrak{t})}\tilde{z}_{0}^{\ee( \mathfrak{t})}. \end{eqnarray}
\begin{proposition}\label{prop:algebraicGz0}
The series $ \mathbf{S}$ satisfies the following algebraic equation:
\begin{eqnarray}\label{eq:algG}
6\,{y}^{2} \left( y-1 \right) {\mathbf{S}}^{3}-\sqrt [3]{2}\sqrt {3}
 \left( x{y}^{2}+12\,\sqrt [3]{2}(y-1) \right) xy{\mathbf{S}}^{2}-
 \left( {y}^{3}\sqrt [3]{2}\sqrt {3}-6\,x{y}^{3}-36\,\sqrt [3]{2}(y-1) \right) {x}^{2}\mathbf{S} \nonumber \\
-\sqrt {3} \left(\sqrt {3}{y}^{2}-2\sqrt {3}y+3/2{y}^{2}+2^{1/3}{x}^{2}{y}^{2}+6\,2^{2/3}x(y-1) \right) y{x}^{3}=0. \quad 
\end{eqnarray}
\end{proposition}

In order to prove Proposition~\ref{prop:algebraicGz0}, we consider a class of triangulations with some decorations on the outer vertices.
We define $\mR$ as the set of maps in $\mS'$ with outer vertices being either \emph{active} or \emph{inactive} and such that
\begin{compactitem}
\item either all the outer vertices are active
\item or the root vertex is active, the other vertex incident to the root edge is inactive, and the active outer-vertices are consecutive along the root face (see for example the left-hand side of Figure~\ref{fig:decomposition-island-bond}).
\end{compactitem}
We denote
$$R(w;x,y,z)=\sum_{ \mathfrak{t}\in \mR}w^{\vinact(\mathfrak{t})}x^{\vact(\mathfrak{t})}y^{\eact(\mathfrak{t})}z^{\ee(\mathfrak{t})},$$
where $\vact(\mathfrak{t})$ and $\vinact(\mathfrak{t})$ are respectively the active and inactive vertices of $\mathfrak{t}$ and $\eact(\mathfrak{t})$ are the number of edges incident to an active outer vertex. Observe that $S(x,y,z)=R(0;x,y,z)$.

\begin{lemma} The series $R(w)\equiv R(w;x,y,z)$ satisfies
\begin{equation}\label{eq:functional-bond}
R(w)= xyz(x+w) +\frac{yz}{w}(R(w)-S)+\frac{yz}{x},R(w)\,S+\frac{yz}{w}(R(w)-S)\,\tS(w),
\end{equation}
 where $S=R(0)$ and $\ds \tS(w)=\sum_{\mathfrak{t}\in\mS'}w^{\length(\mathfrak{t})}z^{\ee(\mathfrak{t})}$. 
\end{lemma}

\begin{proof} This follows from a recursive decomposition of maps in $\mR$ represented in Figure~\ref{fig:decomposition-island-bond}. 
Let $\mathfrak{t}$ be a map in $\mR$. If $\mathfrak{t}$ has a single edge, then it has two vertices and the non-root vertex can be either active or inactive. This gives a contribution of $xyz(x+w)$ to $R(w)$. We now suppose that $\mathfrak{t}$ has several edges. In this case the root edge is incident to an inner triangle $t$, and we denote by $v$ the vertex of $t$ opposite to the root edge. Three situations could occur: 
\begin{compactitem}
\item[(i)] $v$ is an inner vertex of $\mathfrak{t}$,
\item[(ii)] $v$ is an outer vertex of $\mathfrak{t}$ which is active, 
\item[(iii)] $v$ is an outer vertex of $\mathfrak{t}$ which is inactive. 
\end{compactitem}
\fig{width=\linewidth}{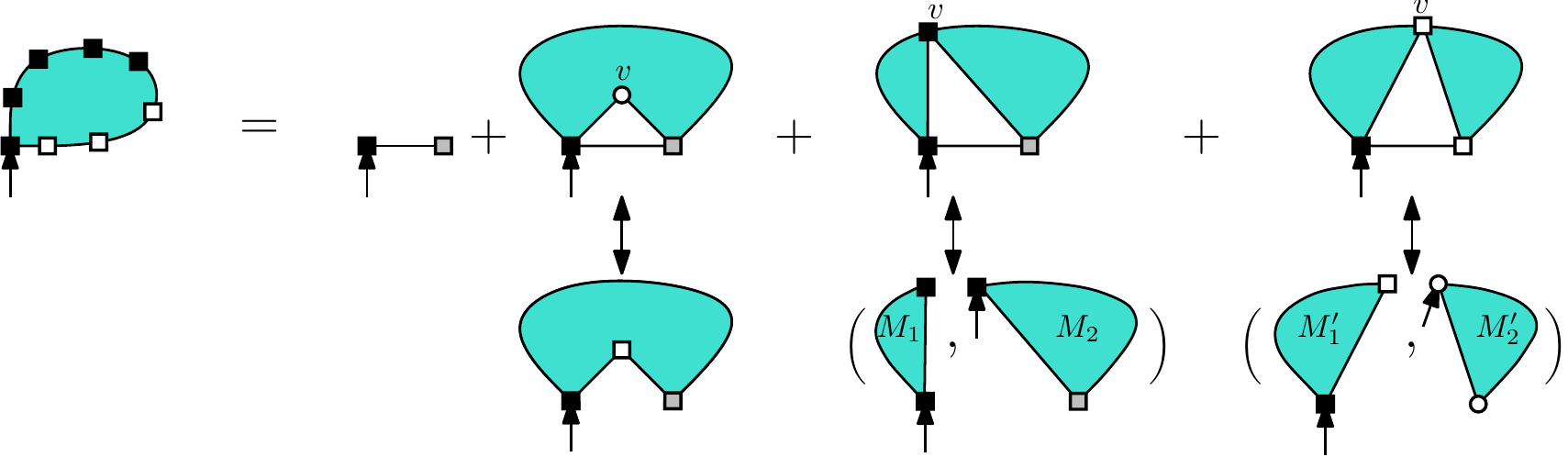}{Recursive decomposition of maps in $\mR$, by deletion of the root edge. The arrows indicate the root-corners. The outer vertices of maps in $\mR$ are indicated by squares colored black if they are active, white if they are inactive, and gray when they could be either.}
Let $\mR^{(i)},\mR^{(ii)},\mR^{(iii)}$ the sets of maps corresponding to cases (i), (ii), and (iii) respectively.
The set $\mR^{(i)}$ contributes $\frac{yz}{w}(R(w)-S)$ to $R(w)$, because deleting the root edge gives a bijection between $\mR^{(i)}$ and the set of maps in $\mR$ having some inactive outer vertices; see Figure~\ref{fig:decomposition-island-bond}. The set $\mR^{(ii)}$ contributes $\frac{yz}{x}\,R(w)\,S$, because deleting the root edge gives a bijection between $\mR^{(ii)}$ and pairs of maps $(\mathfrak{t}_1,\mathfrak{t}_2)\in \mR^2$ such that $\mathfrak{t}_1$ has no inactive outer vertices; see Figure~\ref{fig:decomposition-island-bond}. Lastly, the set $\mR^{(iii)}$ contributes $\frac{yz}{w}(R(w)-S)\,\tS(w)$, because deleting the root edge gives a bijection between $\mR^{(iii)}$ and pairs of maps $(\mathfrak{t}_1',\mathfrak{t}_2')\in \mR\times \mS'$ such that $\mathfrak{t}_1$ has some inactive outer vertices; see Figure~\ref{fig:decomposition-island-bond}. Adding these contributions gives~\eqref{eq:functional-bond}.
\end{proof}

We now complete the proof of Proposition~\ref{prop:algebraicGz0}; details can be found in the accompanying Maple session. 
 We observe that $\tS(w)=w[x^1]R(w;x,1,z)$. Moreover, by specializing~\eqref{eq:functional-bond} to $y=1$ and extracting the coefficient of $x^1$ we obtain
\begin{equation}\label{eq:functional-tS}
\tS(w)=w^2z+\frac{z}{w}(\tS(w)-w\,\tS_1)+\frac{z}{w}\tS(w)^2
\end{equation}
where $\tS_1=[w^1]\tS(w)$. Eliminating $\tS(w)$ between~\eqref{eq:functional-bond} and~\eqref{eq:functional-tS} gives 
\begin{equation}\label{eq0}
A(R(w),S,\tS_1,w,x,y,z)=0,
\end{equation}
where $A:=A(r,s,t,w,x,y,z)$ is the polynomial given by
\begin{eqnarray*}
A&\!\!\!=\!\!\!& 
\big({s}^{2}w{y}^{2}{z}^{2}+xyz(wy+yz-2w)s+{x}^{2} ({w}^{2}{y}^{2}{z}^{2}-t{y}^{2}{z}^{2}+{y}^{2}z-wy-yz+w) \big) r^2 \\
&& -xy \big({s}^{2}yz ( w+z ) -x ( 2 w{z}^{2}xy+2 ty{z}^{2}-2yz+w+z)s-{x}^{2}z ( wy+yz-2 w ) (x+w) \big) r\\
&&-{x}^{2}{y}^{2}z \big( ( {w}^{2}z-tz+1 ) {s}^{2}-x ( w+z )(x+w) s+w{x}^{2}z(x+w)^{2} \big).
\end{eqnarray*}
We proceed to solve this equation using the method suggested in \cite{MBM:quadratic} (the theorem proved there is not directly applicable). We start with an easy claim. 
\begin{claim} There exists a unique formal power series $W(z)\equiv W(y,z)=yz+O(z^2)$ in $\Q(y)[[z]]$ such that 
\begin{equation}\label{eq1}
A_1'(R(W(z)),S,\tS_1,W(z),x,y,z)=0,
\end{equation}
where $A_1'$ denotes the derivative of the polynomial $A$ with respect to its first variable.
\end{claim}

\begin{proof}
First note that one can determine the expansions of $\tS_1$ and $R(w)$ to an arbitrary order using~\eqref{eq:functional-bond} and~\eqref{eq:functional-tS}. Plugging these expansions in~\eqref{eq1} shows that the solutions $W(z)$ of~\eqref{eq1} must satisfy either $W(z)=yz+O(z^2)$ or $W(z)=z+O(z^2)$. After setting $W(z)=yz+z^2\tW(z)$, one sees that~\eqref{eq1} takes the form 
$$\tW(z)=z\frac{Ser(\tW(z),z)}{x^4y^3(y-1)},$$
where $Ser(u,z)$ is a power series in $z$ with coefficients which are polynomial in $u$. This equation easily shows the existence and uniqueness of 
the series $\tW(z)$ (by extracting the coefficient of $z^n$ inductively). This completes the proof of the claim.
\end{proof}

Now observe that, because of~\eqref{eq0} we also have 
\begin{equation}\label{eq2}
A(R(W(z)),S,\tS_1,W(z),x,y,z)=0,
\end{equation}
and because of the derivative of~\eqref{eq0} with respect to $w$ is zero, we also have
\begin{equation}\label{eq3}
A'_4(R(W(z)),S,\tS_1,W(z),x,y,z)=0,
\end{equation}
where $A_4'$ denotes the derivative of $A$ with respect to its fourth variable.
We then use polynomial elimination (e.g. resultant) to eliminate $R(W(z))$ and $W(z)$ from~\eqref{eq1},~\eqref{eq2},~\eqref{eq3}, and obtain a polynomial equation for $S$. Namely, $B(S,\tS_1,x,y,z)=0$, where
\begin{eqnarray*}
B(s,f,x,y,z)&=&-{x}^{5}{y}^{3}{z}^{3}+f{x}^{3}{y}^{3}{z}^{3}-{s}^{2}{x}^{2}{y}^{3}{z}
^{3}+s{x}^{3}{y}^{3}{z}^{2}-s{x}^{2}{y}^{3}{z}^{3}+{s}^{3}{y}^{3}{z}^{2}\\
&&-{x}^{3}{y}^{3}{z}^{2}-{s}^{3}{y}^{2}{z}^{2}-{x}^{4}{y}^{2}z+{x}^{3}
{y}^{2}{z}^{2}-2 {s}^{2}x{y}^{2}z+{x}^{4}yz+2 {s}^{2}xyz+s{x}^{2}y-s{x}^{2}.
\end{eqnarray*}
Finally, we will eliminate $\tS_1$ from this equation. Observe that $\tS_1$ is actually equal to the series denoted $\tT_1$ in the proof of Lemma~\ref{lem:algebraic-site}, hence satisfies $Q(\tS_1,z)=0$, where $Q(u,v)$ is the polynomial given by~\eqref{eq:polT1}. After elimination of $\tS_1$ between these two equations, we obtain an equation of the form 
\begin{equation}\label{eq:S}
C(S,x,y,z)=0, 
\end{equation}
where $C(s,x,y,z)$ is a polynomial of degree 9 in $s$ (see the accompanying Maple session). At $z=\tilde{z}_{0}=(432)^{-1/6}$ this polynomial factorizes as 
$$C(u,x,y,\tilde{z}_{0})={C_1(u,x,y)}^2C_2(u,x,y),$$
where $C_1(u,x,y)$ is equal to 
\begin{eqnarray}
6\,{y}^{2} \left( y-1 \right) {u}^{3}-\sqrt [3]{2}\sqrt {3}
 \left( x{y}^{2}+12\,\sqrt [3]{2}(y-1) \right) xy{u}^{2}-
\left( {y}^{3}\sqrt [3]{2}\sqrt {3}-6\,x{y}^{3}-36\,\sqrt [3]{2}(y-1) \right) {x}^{2}u \nonumber \\
-\sqrt {3} \left(\sqrt {3}{y}^{2}-2\sqrt {3}y+3/2{y}^{2}+2^{1/3}{x}^{2}{y}^{2}+6\,2^{2/3}x(y-1) \right) y{x}^{3}, \nonumber 
\end{eqnarray}
and $C_2(u,x,y)= C_1(u,x,y)+27\sqrt{3}y^3x^3/8$.
Hence, either $C_1(\mathbf{S},x,y)=0$, or $C_2(\mathbf{S},x,y)=0$. By examining the first coefficients (in the variable $y$) of the solutions of these equations, we can rule out the later equation, because it leads to some negative coefficients. Thus we get $C_1(\mathbf{S}(x,y),x,y)=0$, which is precisely~\eqref{eq:algG}. This completes the proof of Proposition~\ref{prop:algebraicGz0}.

\begin{proposition}\label{prop:asymptotic-form-bond} 
For all $y\in[0,1]$, the coefficients of $[x^{n}]\mathbf{S}(x,1-y)$ of $ \mathbf{S}(x,1-y)=S(x,1-y,\tilde{z}_{0})$ form an orthodox sequence with growth constant $\overline{r}(y)$ and exponent $\overline{\beta}(y)$ where $\overline{\beta}$ is defined in Theorem \ref{thm:main} and where for $y\in (\overline{p}_c,1]$ the growth constant $ \overline{r}(y)$ is a root of \begin{eqnarray*}
&& 3 y \left( 23 {(1-y)}^{2}-(6 \sqrt {3}+27)y + 9+48
 \sqrt {3} \right)\, {x}^{2}\\
&&+ 3 {2}^{2/3} \left( 4 \sqrt {3}-5 \right) y(1-y) \left( 3y+2 \sqrt {3} \right)\, x
-2 \sqrt [3]{2}{(1-y)}^{3} \left( 2 \sqrt {3}+9 \right),
\end{eqnarray*}
while for $y\in [0, \overline{p}_c]$ we have $\ds \overline{r}(y)=\frac{2^{2/3}(1-y)}{1+(2\sqrt{3}+5)y}.$ 
\end{proposition}

\begin{proof}
The proof is almost identical to that of Proposition~\ref{prop:asymptotic-form}, only slightly simpler. The interested reader can refer to the accompanying Maple session. Let us simply mention that the discriminant of the algebraic equation~\eqref{eq:algG}, with respect to $ \mathbf{S}$ is 
$\Delta(x)=-36\sqrt[3]{2} x^6y^6 \Delta_1(x)\Delta_2(x)^3$, where
\begin{eqnarray}
\Delta_1(x)&\!\!\!=\!\!\!& 4 {x}^{2}{y}^{3}+2 \sqrt [3]{2} ( 2 \sqrt {3}-3 ) ( -3 y+2 \sqrt {3}+3 ) ( y-1 ) yx\nonumber \\
&&-\frac{{2}^{2/3}}{23} ( 2 \sqrt {3}-9 ) ( 6 \sqrt {3}y+23 {y}^{2}+42 \sqrt {3}+27 y-18 ) ( y-1 ), \label{eq:Delta1-G}\\
\Delta_2(x)&\!\!\!=\!\!\!&2 xy+\frac{\sqrt[3]{2}}{13} ( 2 \sqrt {3}+5 ) ( 13 y-18+2 \sqrt {3} ). \label{eq:Delta2-G}
\end{eqnarray}
As before, the radius of convergence of $ \mathbf{S}$ is solution of $\Delta(x)$, and this leads to the stated equations for $ \overline{r}(y)$.
\end{proof}

We recall that Proposition~\ref{prop:asymptotic-form-bond} together with \eqref{eq:weight-as-coef-bond} immediately implies Proposition~\ref{prop:weight-asymptotic-bond}. 

\subsubsection{An auxiliary generating function}
In the sequel when analyzing bond-percolation, we will also need some information about the asymptotic of the coefficients of the ``dual\footnote{This was not required in the case of site-percolation because after the island and reef decompositions, the two sides of the reef are self-dual and are enumerated via $T$ with mirror parameters $p$ and $1-p$.}'' generating function 
\begin{equation}\label{eq:gfU}
U(x,y,z)=\sum_{ \mathfrak{t}\in\mT}x^{\length( \mathfrak{t})}y^{\eout( \mathfrak{t})}z^{\ee( \mathfrak{t})},
\end{equation}
where we recall that $\mT$ is the set of triangulations with a (non necessarily simple) boundary, and $\eout( \mathfrak{t})$ is the number of edges incident to the outer face. If we specialize to the value $z= \tilde{z}_{0}$ and put $ \mathbf{U}(x,y) := U(x,y, \tilde{z}_{0})$ we can then prove:

\begin{proposition} \label{prop:asymptoF} For any $y \in [0,1]$, the coefficients $[x^{n}] \mathbf{U}(x,y)$ form an orthodox sequence with exponent $ \overline{\beta}(y)$ given by Theorem \ref{thm:main}.
Moreover, for $y \in [ \overline{p}_{c},1]$, the growth constant $\bar{\bar{r}}(y)$ is given by $\ds \frac{5+ 13y - 2 \sqrt{3}}{ \sqrt[3]{2}(10 \sqrt{3}-12)}$ 
while, for $y \in [0, \overline{p}_{c})$ the inverse growth constant $1/\bar{\bar{r}}(y)$ is a root of 
$$(23\,{y}^{2}-6\sqrt {3}y+48\sqrt {3}-73\,y+32 ) y{x}^{2}-2\sqrt [3]{2} ( 5\sqrt {3}-12 ) ( 2\sqrt {3}+3y ) yx+4\,{2}^{2/3} ( 2\sqrt {3}+9)( y-1 ).$$
\end{proposition}

\begin{proof} We first establish an algebraic equation for $U$, hence for $ \mathbf{U}$. The decomposition of triangulations illustrated in Figure~\ref{fig:decomposition-ocean} gives
\begin{equation}\label{eq:decomposition-ocean}
U=1+yx^2zU^2+z\frac{U-1-xU_1}{xy}+(y-1)xz^2U\,(2U-1-2yx^2zU^2)+(y^2-1)x^3yz^3U^3,
\end{equation}
where $U_1=[x^1]U$. 

\fig{width=\linewidth}{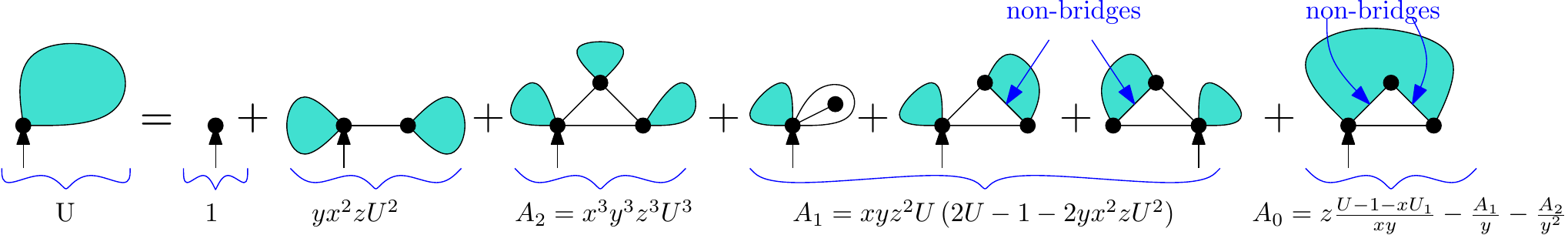}{Recursive decomposition of triangulations in $\mT$ giving~\eqref{eq:decomposition-ocean}. Among the triangulations such that the root edge is not a bridge, we distinguish different cases according to the number of bridges created when deleting the root edge: we denote $A_0,A_1,A_2$ respectively the contribution of the triangulations such that 0, 1 or 2 bridges are created.}

Next, we observe that $U_1=y\tT_1$, where $\tT_1$ is the solution of~\eqref{eq:polT1}. Eliminating $U_1$ between~\eqref{eq:decomposition-ocean} and~\eqref{eq:polT1}, gives an equation of the form $P(U,x,y,z)=0$ for a polynomial $P$. Setting $z=\tilde{z}_{0}$, this equation factorizes and gives an algebraic equation for $ \mathbf{U}$:
\begin{eqnarray}
0={x}^{4}{y}^{2} \left( y-1 \right) ^{2}\,{ \mathbf{U}}^{3}
+2\,\sqrt [3]{2} \left(\sqrt {3}\sqrt [3]{2}y-\sqrt {3}\sqrt [3]{2}+3\,xy \right) y{x}^{2}\,{\mathbf{U}}^{2} \nonumber\\
+{2}^{2/3}\sqrt {3} \left( -{x}^{2}{y}^{2}+{2}^{2/3}\sqrt {3}-6\,\sqrt [3]{2}xy+{x}^{2}y \right) \mathbf{U}
- \left( 2\,\sqrt {3}+3 \right) \left( 4\,\sqrt {3}\sqrt [3]{2}-3\,xy-6\,\sqrt [3]{2} \right). \label{eq:algF}
\end{eqnarray}

Lastly, one can deduce from~\eqref{eq:algF} the asymptotic behavior of $[x^n]\mathbf{U}(x,y)$, for all $y\in[0,1]$. The proof is again along the same line as that of Proposition~\ref{prop:asymptotic-form} (but slightly simpler) and all the computations can be found in the accompanying Maple session. 
 Let us simply mention that the discriminant of the algebraic equation~\eqref{eq:algF}, with respect to $\mathbf{U}$ is 
$\delta(x)=\frac{3(1701+956\sqrt{3})}{50531}x^5y^4 \delta_1(x)\delta_2(x)^3$ where 
\begin{eqnarray}
\delta_1(x)&\!\!\!=\!\!\!\!& (23\,{y}^{2}-6\sqrt {3}y+48\sqrt {3}-73\,y+32 ) y{x}^{2}-2\sqrt [3]{2} ( 5\sqrt {3}-12 ) ( 2\sqrt {3}+3y ) yx\nonumber\\
&&+4\,{2}^{2/3} ( 2\sqrt {3}+9)( y-1 ), \label{eq:Delta1-F}\\
\delta_2(x)&\!\!\!=\!\!\!&\left( 2\sqrt {3}-13y-5 \right) x+10\sqrt {3}\sqrt [3]{2}-12\sqrt [3]{2}. \label{eq:Delta2-F}
\end{eqnarray}
As before, the radius of convergence of $\mathbf{U}$ is solution of $\delta(x)$, and this leads to the stated equations for $\bar{\bar{r}}(p)$. \end{proof}

\subsubsection{Analysis of the derivatives}
We now analyze the asymptotic form of the coefficients of the derivatives of the series ${U}$ and ${S}$ with respect to $z$, evaluated at  $z= \tilde{z}_{0}$. This will be useful to deduce probabilistic estimates on the size of clusters in bond-percolated triangulations (see Section \ref{sec:size}). 

We denote
 \begin{eqnarray} \mathbb{U}(x,y) := \left.\frac{\partial}{ \partial z} U(x,y,z) \right|_{z=\tilde{z}_{0}}\quad \mbox{ and } \quad \mathbb{S}(x,y) := \left. \frac{\partial}{ \partial z} S(x,y,z) \right|_{z=\tilde{z}_{0}}. \end{eqnarray}

\begin{proposition} \label{prop:size-bond} 
For any $y \in [0,1]$, the coefficients $[x^{n}] \mathbb{S}(x,1-y)$ (resp.~$[x^{n}] \mathbb{U}(x,y)$) form an orthodox sequence with the same growth constants as $ [x^{n}]\mathbf{S}(x,1-y)$ (resp.~$ [x^{n}] \mathbf{U}(x,y)$) and with exponent 
$\overline{\gamma}(y)$  equal to $1/2$ except in the critical case where $\overline{\gamma}(\overline{p}_{c}) = 1/3$. 

\end{proposition}

The first step in the proof of Proposition \ref{prop:size-bond} is to get algebraic equations for $\mathbb{S}(x,y)$ and $ \mathbb{U}(x,y)$. 



\begin{lemma}\label{lem:eq-ovG-ovF} The series $\mathbb{S}$ satisfies an algebraic equation of the form
\begin{eqnarray} \label{eq:alg-ovG}
\Delta_1(x)\Delta_2(x)\left((y-1)y^2\,{\mathbb{S}}^{3}-(xy^2-2^{1/3}\,12(y-1))xy {\mathbb{S}}^{2}\right)+B_1(x,y)\,\mathbb{S} + B_0(x,y)=0,
\end{eqnarray}
where $\Delta_1(x)$ and $\Delta_2(x)$ are given by~\eqref{eq:Delta1-G} and~\eqref{eq:Delta2-G} respectively, and $B_0(x,y)$ and $B_1(x,y)$ are polynomials (see Maple session). Similarly, the series $ \mathbb{U}$ satisfies an algebraic equation of the form 
\begin{eqnarray} \label{eq:alg-ovF}
\delta_1(x)\delta_2(x)\left({y}^{2} \left( y-1 \right) ^{2}{x}^{4}{ \mathbb{U}}^{3}-4\,\sqrt {3} \left( y\sqrt {3}\sqrt [3]{2}-\sqrt {3}\sqrt [3]{2}+6\,x
y \right) {x}^{2}y{ \mathbb{U}}^{2}\right)+b_1(x,y)\, \mathbb{U} + b_0(x,y)=0,
\end{eqnarray}
where $\delta_1(x)$ and $\delta_2(x)$ are given by~\eqref{eq:Delta1-F} and~\eqref{eq:Delta2-F} respectively, and $b_0(x,y)$ and $b_1(x,y)$ are polynomials (see Maple session).
\end{lemma}


\begin{proof} Eliminating $S(x,y,z)$ between~\eqref{eq:S} and its derivative with respect to $z$ gives an algebraic equation for $\frac{\partial S}{\partial z}(x,y,z)$ (see Maple session).
Setting $z=\tilde{z}_{0}$ then gives an equation of the form $P_1(\mathbb{S}(x,y),x,y)P_2(\mathbb{S}(x,y),x,y)^2=0,$ where $P_1$ and $P_2$ are polynomials. Moreover we can rule out $P_1(\mathbb{S}(x,y),x,y)$ because it would imply negative coefficients. Hence we get $P_2(\mathbb{S}(x,y),x,y)=0$ which has the form stated in \eqref{eq:alg-ovG}. The proof of \eqref{eq:alg-ovF} is similar.
\end{proof}

Next we prove two lemmas implying that for all $y\in[0,1]$ the radius of convergence of the series $\mathbb{S}(x,y)$ and $ \mathbf{S}(x,y)$ (resp. $ \mathbb{U}(x,y)$ and $ \mathbf{U}(x,y)$) are equal. The first recall a known result  (see \cite[Proposition 9 and Section 6.1]{CLGpeeling}) about the size of the boundary of a critical percolation (a direct derivation by generating function is also provided in the Maple session).

\begin{lemma} \label{lem:bound-nb-edges-Sk}
Let $S_k$ be a random triangulation with simple boundary of length $k$ chosen with probability proportional to $\tilde{z}_{0}^{\ee(S_k)}$. 
There exists a constant $C$ such that for all $k>0$,
\begin{equation}\label{eq:bound-nb-edges-Sk}
 \mathbb{E}[\ee(S_k)]\leq C k^2.
\end{equation}
\end{lemma}


\begin{lemma} \label{lem:bound-nb-edges-bond}
Let $T_k'$ be a random triangulation with simple boundary of length $k$ chosen with probability proportional to $p^{\reef(T_k')}\tilde{z}_{0}^{\ee(T_k')}$.
Let $T_k''$ be a random triangulation with (non necessarily simple) boundary of length $k$ chosen with probability proportional to $p^{\eout(T_k'')}\tilde{z}_{0}^{\ee(T_k'')}$.
There are constants $C'$, $C''$ such that for all $y\in [0,1]$ and all $k>0$, $ \mathbb{E}[\ee(T_k')]\leq C'\,k^2$, and $\mathbb{E}[\ee(T_k'')]\leq C''\,k^2$.
\end{lemma}
\fig{width=.5\linewidth}{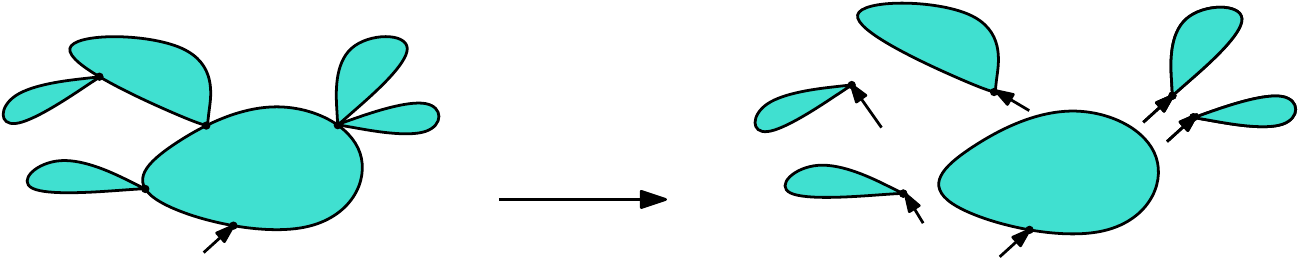}{Decomposition of triangulations with boundary into triangulations with simple boundary.}

\begin{proof}
We first prove the property for $T_k''$. We consider the decomposition of triangulations with boundary into triangulations with simple boundary represented in Figure~\ref{fig:decompose-cuts}. Clearly, $T_k''$ is chosen with probability proportional to $p^{-\textrm{bridge}(T_k'')}\tilde{z}_{0}^{\ee(T_k'')}$, where $\textrm{bridge}(T_k'')$ is the number of components which are just bridges. Moreover, conditional on the number of bridges $b$ and the boundary length $k_1,\ldots,k_l$ of the other components (which must satisfy $2b+\sum_i k_i=k$), each component is chosen independently with probability proportional to $\tilde{z}_{0}^{\#\textrm{edges}}$. Hence \eqref{eq:bound-nb-edges-Sk} implies 
$$\mathbb{E}[\ee(T_k'')]\leq C\max_{b,k_1,\ldots,k_l~|~2b+\sum_i k_i=k}(b+\sum k_i^2)=Ck^2.$$

We now prove the property for $T_k'$. Note that deleting the outer edges and the reef-edges of $T_k'$ we get a union of triangulations with total boundary length at most $\reef(T_k')-2k$. Hence a reasoning similar as before gives $\ds \mathbb{E}[\ee(T_k')|\reef(T_k')=n]\leq Cn^2$. Thus, 
$$\mathbb{E}[\ee(T_k')]=\sum_{n} \mathbb{P}(\reef(T_k')=n)\mathbb{E}[\ee(T_k')|\reef(T_k')=n]\leq C \sum_{n}\mathbb{P}(\reef(T_k')=n)n^2.$$
Moreover, for all $n$, $\mathbb{P}(\reef(T_k')\geq n)$ is maximal for $p=1$. Thus it suffices to show that there exists $D\in \mathbb{R}$ such that for $p=1$ and all $k\geq 0,$ 
\begin{equation}\label{eq:bond-var-reef}
\sum_{n}\mathbb{P}(\reef(T_k')=n)n^2\leq D\,k^2.
\end{equation}
For $p=1$ we have 
$$\sum_{n}\mathbb{P}(\reef(T_k')=n)n^2=\frac{[x^k] \mathbf{S}_{yy}(x,1)+\mathbf{S}_{y}(x,1)}{[x^k]\mathbf{S}(x,1)}.$$
From the equation~\eqref{eq:algG} for $\mathbf{S}(x,y)$, we can deduce (by differentiating with respect to $y$ and polynomial elimination) algebraic equations for $\mathbf{S}_{y}(x,y)$ and $\mathbf{S}_{yy}(x,y)$ (see Maple session). From there it is easy to get the asymptotic behavior of $[x^k]\mathbf{S}_y(x,1)$ and $[x^k]\mathbf{S}_{yy}(x,1)$. This gives 
$$\sum_{n}\mathbb{P}(\reef(T_k')=n)n^2\sim_{k\to \infty} c k^2,$$
for some constant $c>0$ (see Maple session). This implies~\eqref{eq:bond-var-reef}, and completes the proof.
\end{proof}

We can now complete the proof of Proposition~\ref{prop:size-bond}.

\begin{proof}[Proof of Proposition~\ref{prop:size-bond}] We need to determine the asymptotic behavior of $[x^n]\mathbb{S}(x,y)$ and $[x^n] \mathbb{U}(x,y)$. 
We only sketch the process for $\mathbb{S}(x,y)$; the case of $\mathbb{U}(x,y)$ is similar and the details can be found in the accompanying Maple session. First, Lemma~\ref{lem:bound-nb-edges-bond} implies that for all $y\in[0,1]$ the radius of convergence of $\mathbb{S}(x,y)$ and $\mathbf{S}(x,y)$ are equal since $$[x^n]\mathbf{S}(x,y)\leq [x^n]\mathbb{S}(x,y)\leq C'n^2 [x^n]\mathbf{S}(x,y).$$
Moreover, the form of~\eqref{eq:alg-ovG} implies that $\mathbb{S}(x,y)$ is infinite at its radius of convergence. 
Finally, treating the cases $y=1-\overline{p}_c$, $y<1-\overline{p}_c$ and $y>1-\overline{p}_c$ separately, we can apply Newton's method to determine the singular behavior of the series $\mathbb{S}(x,y)$ at its radius of convergence. This translates into the stated properties of $[x^n]\mathbb{S}(x,y)$.
\end{proof}

\section{On admissibility and criticality}
\label{sec:critical}

In this section, we revisit the notions of admissibility and criticality given in the introduction, and give alternative equivalent definitions, some of which appeared in earlier work \cite{Mie08b,Bud15}. Let us fix a weight sequence $\bq$, and recall the definitions of $Z_\bq,Z^\bullet_\bq$ in \eqref{eq:1},~\eqref{eq:2}. 

\begin{proposition}[Characterization of admissibility]
 \label{sec:char-admiss-crit}
For a given weight sequence $\bq$ one has $Z_\bq<\infty$ if and only if $Z^\bullet_\bq<\infty$ (in which case $\bq$ is called \emph{admissible}). 
\end{proposition}

Following \cite{Mie08b}, for $x,y\geq 0$, let
\begin{align*}
 f^\bullet_\bq(x,y)& =\sum_{k,k'\geq 0}x^ky^{k'}\binom{2k+k'+1}{k+1,k,k'}q_{2+2k+k'},\\
 f^\diamond_\bq(x,y) &=\sum_{k,k'\geq 0}x^ky^{k'}\binom{2k+k'}{k,k,k'}q_{1+2k+k'}\, .
\end{align*}
\begin{proposition}
 \label{sec:admiss-crit}
One has $Z^\bullet_\bq<\infty$ if and only if there exists a solution $(x,y)\in \R_+^2$ to the system of equations 
\begin{align}
f_\bq^\bullet(x,y)& =1-\frac{1}{x}\nonumber\\
 f_\bq^\diamond(x,y)&=y\, .\label{eq:4} 
\end{align}
\end{proposition}
In \cite{Bud15}, a slightly weaker result is proved (see also \cite{Mie08b} for a similar statement): it is shown there that $Z^\bullet_\bq<\infty$ if and only if there exists a solution  $(z^+,z^\diamond)$ of \eqref{eq:4} such that 
  \begin{equation}
\label{eq:5}
   (\partial_y+\sqrt{x}\partial_x)f_\bq^\diamond(z^+,z^\diamond)\leq 1\,.
  \end{equation}
It is also proved in \cite{Bud15} that the solution of \eqref{eq:4} satisfying \eqref{eq:5} is unique in this case (see Lemma \ref{claim-solution-4-5} below).
It will turn out from the proof of Proposition \ref{sec:admiss-crit} that this solution of \eqref{eq:4} is characterized by the fact that both its coordinates are minimal. We will show this at the very end of this section. 

The (admissible) weight sequence $\bq$ is then called \emph{critical} (in the sense of \cite{Mie08b,Bud15}) if equality holds in \eqref{eq:5}, and \emph{subcritical} otherwise. The next proposition will show that this notion coincides with the (formally simpler) one given in the introduction. However, before giving the statement, we need another notation. 

If $ \mathfrak{m}$ is a planar map (with at least one face) we denote by $ \mathrm{f_{r}}$ the root face, which is the face adjacent on the right of the root edge. 
Given an admissible weight sequence $ \mathbf{q}$, we introduce the so-called disk partition function 
$$ \mathsf{Disk}_{ \mathbf{q}}^{(\ell)} = \prod_{ \mathfrak{m} \in \mathcal{M}^{(\ell)}} \prod_{ f \in \mathsf{Face}( \mathfrak{m}) \backslash \mathrm{f_{r}} } q_{ \mathrm{deg}(f)},$$ where $\mathcal{M}^{(\ell)}$ is the set of rooted planar maps with a root face of degree $\ell$, and where we noticed that compared to \eqref{eq:boltzmann} the root face is not counted in the product. Here is the main result of this section, which combined with the forthcoming propositions \ref{prop:interface-site} and \ref{prop:interface-bond} completes the proof of our main theorem concerning criticality/subcriticality of the admissible weight sequences $ \dot{ \mathbf{q}}(p)$ and $ \overline{ \mathbf{q}}(p)$. 
\begin{proposition}[Characterizations of criticality] \label{prop:carac-critical} Let $ \mathbf{q}$ be an admissible weight sequence. Then the following conditions are equivalent:
\begin{enumerate}[(i)]
\item $ \mathbf{q}$ is subcritical in the sense of \cite{Mie08b,Bud15}, meaning that there exists a solution $(x,y)$ of \eqref{eq:4} such that $\ds(\partial_y+\sqrt{x}\partial_x)f_\bq^\diamond(z^+,z^\diamond)<1$,
\item $\sum_{\mathfrak{m}}\mathrm{Bolt}_\bq(\mathfrak{m})\, \vv(\mathfrak{m})^2<\infty$, 
\item the sequence $ \mathsf{Disk}_{ \mathbf{q}}^{(\ell)}$ is orthodox with exponent $\alpha = 3/2$.
\end{enumerate}
\end{proposition}

Notice that the fact that the disk partition function is orthodox with exponent $3/2$ for admissible subcritical weight sequences is already used in \cite{BBG12} and in \cite{CurPeccot} but in the case of bipartite Boltzmann maps. The proofs of Propositions \ref{sec:char-admiss-crit} and \ref{prop:carac-critical} are in the spirit of these papers, and relies crucially on the detailed analysis of non-bipartite Boltzmann maps in \cite{Bud15}. Let us first introduce some notation and basic facts that will be useful in both proofs. 

First note that the partition function $Z^\bullet_\bq$ appearing in \eqref{eq:2} can be written as 
$$Z^\bullet_\bq=1+2Z^+_\bq+Z^0_\bq\, ,$$
where 
$$ Z_\bq^+=\sum_{(\mathfrak{m},v)\in \mathcal{M}^+}\mathrm{Bolt}_\bq(\mathfrak{m})\, ,\qquad Z^0_\bq=\sum_{(\mathfrak{m},v)\in \mathcal{M}^0}\mathrm{Bolt}_\bq(\mathfrak{m})\, ,
$$
where $\mathcal{M}^+$ is the set of rooted and pointed maps
$(\mathfrak{m},v)$ such that the root edge $e$ of $\mathfrak{m}$
points toward $v$, in the sense that the graph distance from $e^+$ to
$v$ is strictly smaller than that of $e^-$ to $v$, and $\mathcal{M}^0$
is the analogous set where the two distances are equal. The factor $2$
in front of $Z^+_\bq$ counts rooted and pointed maps whose root edge
links two vertices at different distances from the distinguished
point, and the addition of $1$ allows to take into account the atomic map. The following fact was proved in \cite{Mie08b,Bud15}.
\begin{lemma}[\cite{Mie08b,Bud15}]\label{claim-solution-4-5} If $Z^\bullet_\bq<\infty$, then the unique solution $(z^+,z^\diamond)$ of \eqref{eq:4},~\eqref{eq:5} is given by $z^+=Z^+_\bq+1$ and $z^\diamond=\sqrt{Z^0_\bq}$. 
\end{lemma}

For $g>0$, let $\bq_g$ be the sequence defined by $\bq_g(k)=g^{(k-2)/2}q_k$ for $k\geq 1$. By the Euler formula,
\begin{equation}
 \label{eq:8}
 g\, Z_{\bq_g}^\bullet=\sum_{\mathfrak{m}}g^{\vv(\mathfrak{m})-1}
\vv(\mathfrak{m}) \prod_{f\in\mathsf{Face}(\mathfrak{m})}q_{\deg(f)}\, ,
\end{equation}
so that $gZ_{\bq_g}^\bullet$ is a (possibly infinite) increasing function of $g$. By the same argument, $Z^+_{\bq_g}$ and $gZ^0_{\bq_g}$ are increasing functions of $g$ (we can avoid a multiplication by $g$ in the first case, since the maps $\mathcal{M}^+$ all have at least two vertices), which converge to $Z^+_\bq,Z^0_\bq$ as $g\uparrow 1$ by monotone convergence. We let 
$$x_g=g(Z^+_{\bq_g}+1)\, ,\qquad y_g=\sqrt{gZ^0_{\bq_g}}\, ,$$
which are increasing functions that converge to $x_1=z^+,y_1=z^\diamond$ as $g\uparrow1$. 
Note that $f^\bullet_{\bq_g}(x,y)=f_\bq^\bullet(gx,\sqrt{g}y)$ and
$f^\diamond_{\bq_g}(x,y)=f^\diamond_\bq(gx,\sqrt{g}y)/\sqrt{g}$. 
So if the sequence $\bq_g$ satisfies $Z^\bullet_{\bq_g}<\infty$, then applying Lemma \ref{claim-solution-4-5} to the sequence $\bq_g$ gives
\begin{equation}
 \label{eq:9}
 f^\bullet_\bq(x_g,y_g)=1-\frac{g}{x_g}\, ,\qquad
 f_\bq^\diamond(x_g,y_g)=y_g\, .
\end{equation}
We are now in position to prove the main results of this section. 

\begin{proof}[Proof of Proposition \ref{sec:char-admiss-crit}]
Clearly, $Z^\bullet_\bq<\infty$ implies $Z_\bq<\infty$. The converse is similar to Corollary 23 in \cite{CurPeccot}, which deals with the bipartite case. 
 Assume that $Z_\bq<\infty$. 
As argued e.g.\ in \cite{BBG11,CurPeccot} , one has
\begin{equation}
 \label{eq:12}
  Z_\bq=\int_0^1 g\, Z_{\bq_g}^\bullet \, \mathrm{d}g\, ,
 \end{equation}
 which follows by \eqref{eq:8} and monotone convergence. 
This means that $gZ^\bullet_{\bq_g}<\infty$ for every $g<1$, and so $\bq_g$ is admissible for every $g<1$. Therefore $x_g,y_g$ are solutions of \eqref{eq:9}. 
By taking a monotone limit as $g\uparrow 1$, we get that $(x_1,y_1)=(Z^+_\bq+1,\sqrt{Z^0_\bq})\in [0,\infty]^2$ is a solution of \eqref{eq:4}. 
It is easy to see that when $q_k>0$ for some $k\geq 3$, then any solution of \eqref{eq:4} has finite coordinates, so $Z^\bullet_\bq=2x_1+y_1^2-1<\infty$, and we have proved that $\bq$ is admissible. On the other hand, the case where $q_k=0$ for every $k\geq 3$ is straightforward, since $\mathrm{Bolt}_\bq$ is supported on maps with at most two vertices in this case. 
 \end{proof}

\begin{proof}[Proof of Proposition \ref{prop:carac-critical}]
  Let $\bq$ be an admissible sequence. Due to the easily checked fact that 
\begin{equation}
 \label{eq:7}
 \partial_yf_\bq^\bullet=\partial_xf_\bq^\diamond\, ,\qquad x\partial_x f_\bq^\bullet + f_\bq^\bullet=\partial_y f_\bq^\diamond\, ,
\end{equation}
we deduce from~\eqref{eq:5} that the partial derivatives of $f_\bq^\bullet,f_\bq^\diamond$ of order $1$ have finite limits as $x\to z^+,y\to z^\diamond$ with $x\leq z^+,y\leq z^\diamond$. By convention, for every $(x,y)\in [0,z^+]\times [0,z^\diamond]$ we write e.g.\ $\partial_xf_\bq^\bullet(x,y)$ for the ``left'' limit $\lim_{x'\uparrow x,y'\uparrow y}\partial_xf_\bq^\bullet(x',y')$

It is obvious from \eqref{eq:8} that $\bq_g$ is admissible for every $g\in [0,1]$. Therefore, $(x_g,y_g)$ satisfies \eqref{eq:9}, which we recast as 
\begin{equation}
 \label{eq:10}
 \mathbf{f}(x_g,y_g)=(0,1-g)\, ,
\end{equation}
where $\mathbf{f}(x,y)=(1-f_\bq^\diamond(x,y)/y,1-x+xf_\bq^\bullet(x,y))$. More precisely, $(x_g,y_g)$ is the unique solution of \eqref{eq:10} for which the analog of \eqref{eq:5} holds, that is
$$\partial_y f_\bq^\diamond(x_g,y_g)+\sqrt{x_g}\, \partial_x f_\bq^\diamond(x_g,y_g)\leq 1\, ,$$ and by strict monotonicity, a strict inequality must hold for $g<1$, meaning that $\bq_g$ is always subcritical in the sense of \cite{Mie08b,Bud15}. A little work using \eqref{eq:7} shows that the Jacobian of $\mathbf{f}$ at the point $(x_g,y_g)$ is given by 
$$\frac{1}{y_g}\left((1-\partial_yf_\bq^\diamond (x_g,y_g))^2-(\sqrt{x_g}\partial_xf_\bq^\diamond (x_g,y_g))^2\right)\, ,$$
which is non-zero for every $g<1$, and also for $g=1$ if and only if $\bq$ is subcritical. 

Due to this discussion, the implicit functions theorem applies and shows that $(x_g,y_g)$ is continuously differentiable in the variable $g\in (0,1)$. By differentiating \eqref{eq:10}, we obtain after some algebra that for $g<1$, 
$$\left(\frac{g-1}{x_g}+\frac{(1-\partial_yf_\bq^\diamond(x_g,y_g))^2-(\sqrt{x_g}\partial_xf_\bq^\diamond(x_g,y_g))^2}{1-\partial_yf_\bq^\diamond(x_g,y_g)}\right)x'_g=1\, ,$$
where we note that the denominator $1-\partial_yf_\bq^\diamond(x_g,y_g)$ is strictly positive, even when $g=1$, by \eqref{eq:5}. Therefore, by taking a limit as $g\uparrow 1$, we see that $x'_1=\infty$ if and only if $\bq$ is critical in the sense of \cite{Mie08b,Bud15} (and in this case, $y'_1=\infty$ as well). 
Moreover, Lemma \ref{claim-solution-4-5} gives $Z^\bullet_\bq=x_1+y_1^2-1$. Thus,  the (left-)derivative of $g\mapsto Z^\bullet_{\bq_g}$ at $g=1$ is infinite if and only if $\bq$ is critical. 
It is immediate by \eqref{eq:8} to see that it is equivalent to \eqref{eq:11}. This proves the equivalence between (i) and (ii). 

Now, in order to study the asymptotic of $\mathsf{Disk}^{(\ell)}_\bq$, we use
ideas from \cite{CurPeccot,Bud15}. By Proposition 2 in \cite{Bud15} (and
the discussion in Section 3.2 of this paper), one can use the pointed
analog $\mathsf{Disk}^{\bullet,(\ell)}_\bq$ of $\mathsf{Disk}^{(\ell)}_\bq$,
defined by
$$ \mathsf{Disk}_{ \mathbf{q}}^{\bullet,(\ell)} = \prod_{ \mathfrak{m}
 \in \mathcal{M}^{(\ell)}} \prod_{ f \in \mathsf{Face}( \mathfrak{m})
 \backslash \mathrm{f_{r}} } q_{ \mathrm{deg}(f)}\, .$$
Following \cite{Bud15}, one has
$$\mathsf{Disk}^{\bullet,(\ell)}_{\bq}=(c_+)^\ell h^{(0)}_r(\ell)\, ,$$
where $c_\pm=z^\diamond\pm 2\sqrt{z^+}$, $r=-c_-/c_+$, and 
$$h^{(0)}_r(\ell)=\frac{1}{4^\ell}\binom{2\ell}{\ell}\sum_{n=0}^\ell\frac{\binom{2n}{n}\binom{2\ell-2n}{\ell-n}}{\binom{2\ell}{\ell}}(-r)^n$$ 
satisfies the asymptotic $h^{(0)}_r(\ell)\sim
1/\sqrt{\ell\pi(1+r)}$ as $\ell\to\infty$, uniformly for $r$ varying
in compact subsets of $(-1,1)$. 
A reasoning similar to that leading to \eqref{eq:12} gives 
$$\mathsf{Disk}^{(\ell)}_{\bq}=\int_0^1\mathrm{d}g \, g^{\ell/2}\,
\mathsf{Disk}^{\bullet,(\ell)}_{\bq_g}=\int_0^1\mathrm{d}g\,
(\sqrt{g}c_+(g))^\ell h^{(0)}_{r(g)}(\ell)\, .$$
Assuming that $\bq$ is subcritical, one has
$c_+-c_+(g)=(1-g) \cdot A_g$ for some continuous $A_g$ converging to
$x'_1+y'_1$ as $g\uparrow 1$. This remains true in the critical case,
but since the derivatives explode one has
$A_g\to\infty$ as $g\uparrow 1$. 
Note that $K=\{r(g):g\in [0,1]\}$ is a compact subset of $(-1,1)$ (the
value $1$ being attained only in the bipartite case, which is easier
and implicitly excluded here). Therefore, we obtain
\begin{equation}
 \label{eq:3}
 \mathsf{Disk}^{(\ell)}_\bq\sim
(c_+)^\ell\sqrt{\frac{1}{\ell\pi}}\int_0^1\mathrm{d} g\, (1-(1-g)B_g)^\ell \frac{1+\eta(r(g),g)}{\sqrt{1+r(g)}}\, ,
\end{equation}
where $B_g$ converges (as $g\to 1$) to a finite constant $B>0$ if $\bq$ is subcritical, and to $\infty$ if $\bq$ is critical, and $\sup_{r\in K}|\eta(r,g)|$ has limit $0$ as $g\to1$. An application of Laplace's method entails that $\mathsf{Disk}^{(\ell)}_\bq$ is orthodox with exponent $3/2$ when $\bq$ is subcritical. Note that in any case, even if $\bq$ is critical, the radius of convergence of the generating series $\sum_\ell \mathsf{Disk}^{(\ell)}_\bq z^\ell$ is equal to $c_+$ by (14) in \cite{Bud15}, so that 
 $\limsup_{\ell\to\infty}(\mathsf{Disk}^{(\ell)}_\bq) ^{1/\ell}=c_+$. However, when $\bq$ is critical, the Laplace method applied to \eqref{eq:3} shows that $\mathsf{Disk}^{(\ell)}_\bq=(c_+)^\ell\phi(\ell)$ with $\phi(\ell)=o(\ell^{-3/2})$. Putting these two facts together shows that $\mathsf{Disk}^{(\ell)}_\bq$ cannot be orthodox with exponent $3/2$. 
 \end{proof}

 \begin{proof}[Proof of Proposition \ref{sec:admiss-crit}]
The proof is mainly inspired from \cite{CLGMcactus}, which dealt with finitely supported $\bq$. 
  We already know from \cite{Mie08b} that if $\bq$ is admissible, then \eqref{eq:4} has a solution. Conversely, let us assume that \eqref{eq:4} has a solution $(x_0,y_0)$. We need to show that there exists a (possibly different) solution that also satisfies \eqref{eq:5}.
   
  To avoid trivialities, let us assume that there exists an odd integer $k\geq 3$ such that $q_k>0$, the bipartite case being well-studied, and the case where only $q_1$ and $q_2$ are positive being trivial. We assume that there exists some $(x_0,y_0)\in \R_+^2$ such that $f_\bq^\bullet(x_0,y_0)=1-1/x_0$ and $f_\bq^\diamond(x_0,y_0)=y_0$, and note that necessarily $y_0>0$ and $x_0>1$, because $f_\bq^\diamond(x,0)>0$ for every $x>0$ and $f_\bq^\bullet(1,y)>0$ for every $y>0$. 

Let us set $G(x,y)=f_\bq^\bullet(x,y)-1+1/x$ and $H(x,y)=f_\bq^\diamond(x,y)-y$, which defines two analytic functions on $(0,x_0)\times (0,y_0)$ such that 
\begin{itemize}
\item $G$ is strictly convex in $x$, and increasing in $y$
\item $H$ is strictly convex in $y$, and increasing in $x$. 
\end{itemize}
Let $y\in (0,y_0)$. Then $G(x_0,y)<0$ since $G(x_0,\cdot)$ is increasing. Since $G(1,y)>0$ and by convexity of $G(\cdot,y)$, there exists a unique $\phi(y)\in (1,x_0)$ such that $G(\phi(y),y)=0$, and since $G(\phi(y),y')<0$ for every $y'<y$, it follows that $\phi$ is a strictly increasing function. For the same reason, there exists a strictly increasing function $\psi$ on $(1,x_0)$ such that $H(x,\psi(x))=0$. 
Being increasing, they admit continuous extensions to $[0,y_0]$ and $[1,x_0]$ respectively, and one has 
$$\{G=0\}\cap ([0,x_0]\times [0,y_0])=\{(\phi(y),y):y\in [0,y_0]\}\cup\{(x_0,y_0)\}\, ,$$
and similarly for $\{H=0\}$. 

By analyticity of $G,H$ in $(0,x_0)\times(0,y_0)$, the implicit function theorem shows that $\phi,\psi$ are also analytic in this domain, and 
$$\phi'(y)=-\frac{\partial_yG}{\partial_x G}\geq 0\, ,\qquad \psi'(x)=-\frac{\partial_xH}{\partial_y H}\geq 0\, ,$$
the partial derivatives of $G,H$ being respectively evaluated at $(\phi(y),y)$ and $(x,\psi(x))$. Since clearly $\partial_yG\geq 0$ and $\partial_xH\geq 0$, this entails that $\partial_xG\leq 0$ and $\partial_yH\leq 0$ along the graphs of $\phi,\psi$ respectively. 
Taking a second derivative then gives 
$$\phi''=-\frac{(\phi')^2\partial_{xx}G+2\phi'\partial_{xy}G+\partial_{yy}G}{\partial_xG}\geq 0\, ,$$
and similarly for $\psi$, so that $\phi,\psi$ are convex functions, as well as their respective extensions to $[1,x_0]$ and $[0,y_0]$. 

By convexity, the graphs $\{(\phi(y),y):y\in (0,y_0]\}$ and $\{(x,\psi(x)):x\in (1,x_0]\}$ necessarily intersect 
\begin{itemize}
\item either at exactly one point in $[1,x_0]\times [0,y_0]$
\item or exactly at two points, one in $[1,x_0)\times [0,y_0)$ and the other being $(x_0,y_0)$. 
\end{itemize}
Let $(x_m,y_m)$ be this intersection point, which in the second case is chosen to be the one lying in $[1,x_0)\times [0,y_0)$. Since $(x_0,y_0)$ was initially chosen to be any solution of $G=H=0$, we see that for any such solution $(x,y)$ different from $(x_m,y_m)$, one has $x_m< x$ and $y_m< y$. Moreover, since we assumed that $H(1,0)>0$, we must have $x_m>1$ and $y_m>0$.  

Again by convexity of $\phi,\psi$, one can see that at the point $(x_m,y_m)$, one has $\det(\nabla G,\nabla H)\geq 0$. Here, one should be careful to define the gradients by taking left-limits in the case where $(x_m,y_m)=(x_0,y_0)$, and note that the determinant vanishes if and only if the curves $G=0,H=0$ are tangent at $(x_m,y_m)$. 
By using \eqref{eq:7}, it is easy to see that this inequality boils down to 
$$(1-\partial_yf_\bq^\diamond-\sqrt{x}\partial_xf_\bq^\diamond)(1-\partial_yf_\bq^\diamond+\sqrt{x}\partial_xf_\bq^\diamond)\geq 0$$
at the point $(x_m,y_m)$. But since (still at this point) $1-\partial_y f_\bq^\diamond=-\partial_yH\geq 0$, we deduce that 
$1-\partial_yf_\bq^\diamond-\sqrt{x}\partial_xf_\bq^\diamond\geq 0$, and this is exactly \eqref{eq:5}, showing that $\bq$ is admissible, as wanted. 
 \end{proof}

\section{Applications}
We now turn to applications of our results. In particular we compute the tail distribution of the length of a typical percolation interface. By relating the later to the disk partition function, we are able, using our new criticality criterion (Proposition \ref{prop:carac-critical}) to prove that the clusters are subcritical Boltzmann maps if and only if $p < p_{c}$. We also compute the size of the hull of percolation clusters conditioned on having a large boundary and recover the phenomenology of \cite{CKperco}. Last but not least, we show that our results can easily be transferred to the infinite setting of the UIPT yielding to an new way of computing the critical percolation thresholds.

\subsection{Behavior of interface, cluster size, and disk partition function}
\label{sec:behinterface}
We start with the site-percolation case. Fix $\ell \geq 4$ and $p \in [0,1]$. We write $ \dot{\mathcal{L}}(\ell,p)$ for the event on which the cluster $ \dot{ \mathfrak{C}}(p)$ has a root-face of degree $\ell$ (recall that the root face of a map is the face adjacent to the root edge on its right). Recall that we imposed that both endpoints of the root edge are black. Hence  on the event $ \dot{\mathcal{L}}(\ell,p)$, for $\ell\geq 4$, the third vertex of the root face of the percolated triangulation  is always white. By the island decomposition of Section \ref{sec:island-site}, the event $ \dot{\mathcal{L}}(\ell,p)$ happens if and only if the underlying percolated triangulation is obtained by gluing a triangulation with a general boundary of perimeter $\ell$ with all external vertices colored black onto a site-island with (simple) boundary of perimeter $\ell$ whose external vertices are also black.
Using the notation of \eqref{def:dotw} and performing similar calculations as in \eqref{eq:relatingWktoF}, it follows that the $ \mathbf{q}_{0}$-Boltzmann weight of the event $ \dot{\mathcal{L}}(\ell,p)$ is
 \begin{eqnarray} \mathrm{Bolt}_{ \mathbf{q}_{0}}(\dot{\mathcal{L}}(\ell,p)) & = & \sum_{ \mathfrak{i} \in \dot{\mathcal{I}}_{\ell}} p^{\vopen( \mathfrak{i})}(1-p)^{\vclosed( \mathfrak{i})}z_{0}^{\fin( \mathfrak{i})} \sum_{ \begin{subarray}{c}\mathfrak{t} \in \mathcal{T} \\ \mathrm{Length}( \mathfrak{t}) = \ell \end{subarray}} z_{0}^{\fin( \mathfrak{t})} p^{\vout( \mathfrak{t})}\nonumber \\
 & = & \dot{W}_{\ell}(p) \times z_{0}^{-\ell/3} [x^{\ell}]T(x,p,\tilde{z}_{0}). \label{eq:disksite} \end{eqnarray}
 Notice also, that since $ \dot{ \mathfrak{C}}(p)$ is $ \dot{ \mathbf{q}}(p)$-Boltzmann distributed, we also have by the very definitions of the Boltzmann measure and the disk partition function:
 \begin{eqnarray} 
\mathrm{Bolt}_{ \mathbf{q}_{0}}(\dot{\mathcal{L}}(\ell,p)) \propto \mathsf{Disk}_{\dot{\mathbf{q}}(p)}^{(\ell)} \times \dot{q}_{\ell}(p).
\label{eq:diskboltsite}
\end{eqnarray}
The following proposition (and its analog Proposition \ref{prop:interface-bond} in the bond-percolation case) together with Proposition \ref{prop:carac-critical} completes the proof of our Theorem \ref{thm:main}:

\begin{proposition}\label{prop:interface-site} \hspace{0mm}
 \begin{enumerate}
 \item 
The probability that the degree of the root face of $ \dot{ \mathfrak{C}}(p)$ is equal to $\ell$ decreases as $\ell^{-10/3}$ if $p= \dot{p}_{c}=1/2$ and decreases exponentially fast otherwise. 
\item
The disk partition function $ \mathsf{Disk}^{(\ell)}_{ \dot{ \mathbf{q}}(p)}$ is orthodox with exponent $3/2$ if $p \in [0, \dot{p}_{c})$, with exponent $5/3$ if $p= \dot{p}_{c}$, and with exponent $5/2$ if $ p \in ( \dot{p}_{c},1]$. 
\item 
When $p\in [0,\dot{p}_c)$, the tail distribution of the number of vertices of $\dot{\mathfrak{C}}(p)$ decreases exponentially. 
\item 
When $p\in (\dot{p}_c,1]$, we have $ \mathbb{P}( \vv( \dot{\mathfrak{C}}(p)) \geq n) \sim c \ n^{-3/2}$ for some $c>0$ (depending on $p$). 

\end{enumerate}
\end{proposition}

\proof 
Using Proposition \ref{prop:weight-asymptotic-site} and Proposition \ref{prop:asymptotic-form} to compute  the asymptotic of the right-hand side of \eqref{eq:disksite}, we obtain
\begin{equation}\label{eq:assympt-interface} 
\mathrm{Bolt}_{ \mathbf{q}_{0}}(\dot{\mathcal{L}}(\ell,p)) \propto  \rho(p)^{\ell} \ell^{ - \dot{\beta}(p) - \dot{\beta}(1-p)},
\end{equation}
where $\ds \rho(p)= \frac{\tilde{z}_{0}\dot{r}(1-p)}{1- \tilde{z}_{0} \dot{r}(p)}$ and where $\dot{r}(p)$ is defined in Proposition \ref{prop:asymptotic-form}.
A resultant computation shows that the growth constant  $\rho(p)$ is different from $1$ (hence smaller than $1$ since we are dealing with a probability distribution) when $p \ne \dot{p}_{c}$ and is easily seen to be equal to $1$ in the case $p = \dot{p}_{c}$. This proves the first point of the proposition. 

The second point follows by comparing \eqref{eq:diskboltsite} and \eqref{eq:assympt-interface}: since by Theorem \ref{thm:main} the sequence $ \dot{\mathbf{q}}(p)$ is orthodox with exponent $\dot{\beta}(p)$ necessarily $\mathsf{Disk}_{\dot{\mathbf{q}}(p)}^{(\ell)}$ is orthodox with exponent $\dot{\beta}(1-p)$.

Let us come to point 3. 
Let $p\in [0,\dot{p}_c)$ and let $\bq=\dot{\bq}(p)$ for simplicity. Note that from point 2.\ and Proposition \ref{prop:carac-critical}, the weight sequence $\bq$ is subcritical. 
We use notation from Section \ref{sec:critical} and rewrite, for $g\geq 1$
\begin{align*}
f^\diamond_{\bq}(gz^+,\sqrt{g}z^\diamond)&=\sum_{\ell\geq 0}q_{1+\ell}\, g^{\ell/2}\sum_{2k+k'=\ell}\binom{\ell}{k,k,k'}(z^+)^k (z^\diamond)^{k'}\\
&=\sum_{\ell\geq 0}q_{\ell+1} \, g^{\ell/2}\mathsf{Disk}^{\bullet,(\ell)}_\bq\, .
\end{align*}
where we used the representation of $\mathsf{Disk}^{\bullet,(\ell)}_\bq$ given in \cite[page 31]{Bud15}. Since $\mathsf{Disk}^{\bullet,(\ell)}_\bq$ and $\mathsf{Disk}^{(\ell)}_\bq$ have the same growth constant, we deduce from \eqref{eq:diskboltsite} and \eqref{eq:assympt-interface} that the above sum converges for every $g\leq 1/\rho(p)$.
Since $p<\dot{p}_c$, we have $1/\rho(p)>1$ and therefore $f_\bq^\diamond(gz^+,\sqrt{g}z^\diamond)<\infty$ for some $g>1$. Clearly, this implies that
$f_\bq^\bullet(gz^+,\sqrt{g}z^\diamond)<\infty$ as well because of the identities \eqref{eq:7}. Using the fact that $\bq$ is subcritical, we can then solve \eqref{eq:10} in an open neighborhood of $g=1$ by using the implicit function theorem, and this shows that $\bq_g$ is admissible for some $g>1$. This means that 
$$g^2Z^\bullet_{\bq_g}=\sum_{\mathfrak{m}}g^{\vv(\mathfrak{m})} \vv( \mathfrak{m})\mathrm{Bolt}_\bq(\mathfrak{m})<\infty,$$
for some $g>1$, as wanted. 

Finally, we prove point 4. Let $p>\dot{p}_c$ be fixed. By point 2 and Proposition \ref{prop:carac-critical}, we know that $\dot{\mathfrak{C}}(p)$ is a critical Boltzmann map, and by point 3 that the root face has an exponential tail. This is one way to state that it is a {\em regular critical} Boltzmann map, as defined in \cite{Mie06}. From this, one concludes that the tail distribution for the number of vertices is given by 
$$\P{\vv(\dot{\mathfrak{C}}(p))\geq n}\sim c\, n^{-3/2}\, ,$$
for some $c\in (0,\infty)$. This was already implicitly used in \cite{Mie08b} or \cite[Section 6]{St14} and is an easy consequence of
 \begin{itemize}
\item the Bouttier-Di Francesco-Guitter bijection, which allows to describe Boltzmann maps in terms of certain multitype Galton-Watson trees, in which the vertices of a particular type correspond bijectively to the vertices of the map, \item a classical estimation (see for instance Lemma 6 in \cite{Mie08c}) on the probability that a (multi type) critical Galton-Watson trees with a finite variance has at least $n$ vertices (of a given type). Notice that the criticality and the finite variance condition is guaranteed by the condition of {\em regular} criticality of the underlying Boltzmann map.
\end{itemize}

 \endproof

We now move to the case of bond-percolation. Fix $\ell \geq 4$ and $p \in [0,1]$. We write $ \overline{\mathcal{L}}(\ell,p)$ the event on which the degree of the root face of $ \overline{ \mathfrak{C}}(p)$ has degree $\ell$. 
Applying the island decomposition of Section \ref{sec:island-bond}, we see that the event $ \overline{\mathcal{L}}(\ell,p)$ happens if and only if the underlying percolated triangulation is obtained by gluing a triangulation with a general boundary of perimeter $\ell$ with all external edges colored black onto a bond-island with (simple) boundary of perimeter $\ell$ whose external edges are also black. Using the notation of Section \ref{sec:island-bond} and performing the same kind of calculations it follows that
 \begin{eqnarray*} \mathrm{Bolt}_{ \mathbf{q}_{0}}(\overline{\mathcal{L}}(\ell,p)) & = & \sum_{ \mathfrak{i} \in \overline{\mathcal{I}}_{\ell}} p^{\eopen( \mathfrak{i})}(1-p)^{\eclosed( \mathfrak{i})}z_{0}^{\fin( \mathfrak{i})} \sum_{ \begin{subarray}{c}\mathfrak{t} \in \mathcal{T}\\ \mathrm{Length}( \mathfrak{t})=\ell \end{subarray}} z_{0}^{\fin( \mathfrak{t})} p^{\eout( \mathfrak{t})}\\
 & = & \overline{W}_{\ell}(p) \times z_{0}^{-\ell/3} [x^{\ell}]U(x,p,\tilde{z}_{0}). \end{eqnarray*}
 Moreover, since $ \overline{ \mathfrak{C}}(p)$ is $ \overline{ \mathbf{q}}(p)$-Boltzmann distributed, we also have
 $$ \mathrm{Bolt}_{ \mathbf{q}_{0}}(\overline{\mathcal{L}}(\ell,p)) \propto \mathsf{Disk}_{\overline{\mathbf{q}}(p)}^{(\ell)} \times \overline{q}_{\ell}(p).$$
 
\begin{proposition}\label{prop:interface-bond}  \hspace{0mm}
 \begin{enumerate}
 \item 
The probability that the degree of the root face of $ \overline{ \mathfrak{C}}(p)$ is equal to $\ell$ decreases as $\ell^{-10/3}$ if $p= \overline{p}_{c}=\frac{2 \sqrt{3}-1}{11}$ and decreases exponentially fast otherwise. 
\item The disk partition function $ \mathsf{Disk}^{(\ell)}_{ \overline{ \mathbf{q}}(p)}$ is orthodox with exponent $3/2$ if $p \in [0, \overline{p}_{c})$, with exponent $5/3$ if $p= \overline{p}_{c}$, and with exponent $5/2$ if $ p \in ( \overline{p}_{c},1]$.
\item 
When $p\in [0,\overline{p}_c)$, the tail distribution of the number of vertices of $\overline{\mathfrak{C}}(p)$ decreases exponentially. 
\item When $p\in (\overline{p}_c,1]$, we have $ \mathbb{P}( \vv( \overline{\mathfrak{C}}(p)) \geq n) \sim c \ n^{-3/2}$ for some $c>0$ (depending on $p$). 
\end{enumerate}
\end{proposition}
\proof The proof is similar to that of Proposition \ref{prop:interface-site} and uses Proposition \ref{prop:weight-asymptotic-bond} and Proposition \ref{prop:asymptoF}. We leave the details to the reader.
\endproof

%

\subsection{Sizes of hulls}

\label{sec:size}
In this section we show that the total size of the hull of the origin cluster behaves differently in the subcritical, critical and super-critical phases. More precisely we denote by $ \dot{\mathfrak{H}}(p)$ and $\overline{ \mathfrak{H}}(p)$ the hulls of the origin percolation clusters obtained respectively from $ \dot{\mathfrak{C}}(p)$ and $ \overline{ \mathfrak{C}}(p)$ by filling-in all the faces of the cluster except from the root face. In other words, $ \dot{\mathfrak{C}}(p)$ and $ \overline{ \mathfrak{C}}(p)$ are the parts of percolated triangulation on one side of the percolation interface at the root. We are interested in the expected number of edges of these submaps, as we condition the percolation interface at the root to be long.

Recall the definition of the event $ \dot{\mathcal{L}}(\ell,p)$ from Section \ref{sec:behinterface}. 
\begin{proposition}[Size of the hull of a large cluster]\label{prop:size-island-site} The number of edges of the  hull of the origin cluster satisfies
$$ \mathbb{E}\big[ \ee(\dot{ \mathfrak{H}}(p)) \big| \dot{\mathcal{L}}(\ell,p)\big]  \underset{\ell \to \infty}{\sim} \dot{c}(p) \ell^{\dot{\delta}(p)},$$ where $\dot{c}(p)>0$ and $\dot{\delta}(p)=1$ in the subcritical phase $p\in [0, \dot{p}_{c})$, in the critical case $\dot{\delta}( \dot{p}_{c})=4/3$, and $\dot{\delta}(p)=2$ in the supercritical phase $p\in( \dot{p}_{c},1]$.
\end{proposition}

\begin{proof} On the event $ \dot{ \mathcal{L}}(\ell,p)$ the hull $\dot{ \mathfrak{H}}(p)$ of the origin cluster is simply a triangulation with a boundary of perimeter $\ell$ and sampled according to $p^{\vout (\mathfrak{t})} \tilde{z}_{0}^{ \ee( \mathfrak{t})}$. It follows readily that the conditional expectation in the proposition is proportional to (we do not count the normalization factors) 

$$ \mathbb{E}[\ee(\dot{ \mathfrak{H}}(p)) \mid \dot{\mathcal{L}}(\ell,p)] \quad \propto \quad 
 \frac{[x^{\ell}] \frac{\partial}{\partial z}{T}(x,p,z)|_{z= \tilde{z}_{0}}}{ [x^{\ell}] T(x,p, \tilde{z}_{0})}=
\frac{[x^{\ell}] \mathbb{T}(x,p)}{ [x^{\ell}] \mathbf{T}(x,p)}.$$
The result then follows by combining Proposition \ref{prop:asymptotic-form} and Proposition \ref{prop:size-site}.
\end{proof}

\begin{rek} The above result is in agreement with \cite[Theorem 1.2]{CKperco}. Specifically, when conditioning a subcritical cluster to have a very large root face, then this face in fact chooses the geometry of a tree. In this scenario, the hull of the cluster is obtained by filling-in small holes and thus the total size is roughly proportional to the perimeter of the root face, hence $\dot{\delta}(p) = 1$ when $p \in [0, \dot{p}_{c})$. In the supercritical phase, the easiest way for the origin cluster to have a large face is when the later has very few pinch points at large scale (it is almost ``simple''). The hull of the cluster is thus obtained by filling-in an essentially unique simple hole of perimeter $\Theta(\ell)$ with a generic triangulation of size $\ell^{2}$.
\end{rek}


As expected, a similar result holds in the case of bond percolation, and the proof is mutatis mutandis the same as that of Proposition \ref{prop:size-island-site} using the functions $ \mathbf{U}$, $ \mathbb{U}$ and Propositions \ref{prop:size-bond} and \ref{prop:asymptoF} instead of the functions $ \mathbf{T}$ and $ \mathbb{T}$ and Propositions \ref{prop:size-site} and \ref{prop:asymptotic-form}. 
\begin{proposition}[Size of the hull of a large cluster]\label{prop:size-island-bond} The number of edges of the  hull of the origin cluster satisfies
$$ \mathbb{E}\big[\ee(\overline{ \mathfrak{H}}(p)) \big| \overline{\mathcal{L}}(\ell,p)\big]  \underset{\ell \to \infty}{\sim} \overline{c}(p) \ell^{\overline{\delta}(p)},$$ where $\overline{c}(p)>0$ and $\overline{\delta}(p)=1$ in the subcritical phase $p\in [0, \overline{p}_{c})$, in the critical case $\overline{\delta}( \overline{p}_{c})=4/3$, and $\overline{\delta}(p)=2$ in the supercritical phase $p\in( \overline{p}_{c},1]$.
\end{proposition}

\subsection{Links with percolation on the UIPT}\label{sec:links-with-perc}
We now turn our attention to percolation models on the type-I Uniform Infinite Planar Triangulation (UIPT), which was introduced in \cite{AS03} and can be obtained as the local limit as $n\to\infty$ of a critical Boltzmann triangulation $M$ conditioned on $|M|>n$. This means that if $B_r(M)$ denotes the combinatorial ball of radius $r$ centered at the root edge of $M$ (i.e. the map obtained by keeping only those faces which have at least a vertex at graph distance less than $r-1$ from the origin of the root edge), then this converges in distribution (for the discrete topology) to a limiting map $B_r(M_\infty)$, which one interprets as the ball of radius $r$ of an infinite triangulation of the plane $M_\infty$. See \cite{AS03} for details. The local convergence generalizes in an obvious way to the (site or bond) percolation models on triangulations, where the convergence now deals with maps in which the vertices or edges are colored. 

For $p \in [0,1]$, we let $\dot{\mathfrak{C}}_\infty(p),\overline{\mathfrak{C}}_\infty(p)$ be the site/bond percolation cluster of the root edge in $M_\infty$, which is now a finite or infinite submap of $M_\infty$. We call (annealed) \emph{site-percolation threshold} of the  UIPT the minimal value of $p\in[0,1]$ above which the origin cluster of the UIPT has a positive probability to be infinite:
 $$\dot{p}_c(\mathrm{UIPT})=\inf\{p\geq 0:\mathbb{P}(|\dot{\mathfrak{C}}_\infty(p)|=\infty)>0\}.$$
The \emph{bond-percolation threshold} of the  UIPT, is defined similarly and is denoted $\overline{p}_c(\mathrm{UIPT})$. In \cite{Ang03}, it was proved that $\dot{p}_c(\mathrm{UIPT})=1/2$. 

\begin{rek} Notice that in the above definition the probability $ \mathbb{P}$ averages in the same time over the choice of the map and that of the percolation. We could have defined a quenched site-percolation threshold by putting 
$$ \dot{ \mathsf{p}}_{c} =\inf \{p\geq 0: \mbox{almost surely with respect to }M_{\infty} \mbox{ we have } \mathbb{P}(|\dot{\mathfrak{C}}_\infty(p)|=\infty)>0\}, $$ where now the probability $ \mathbb{P}$ only averages over the percolation. It was however argued in \cite{Ang03} that the two definitions coincide in the case of site percolation on the UIPT, and this generalizes easily to bond percolation. We shall then make no difference in the sequel between quenched and annealed percolation thresholds.
\end{rek}

 Angel and Curien \cite{ACpercopeel} proved that $\overline{p}_c(\mathrm{UIPT})=(2\sqrt{3}-1)/11$, in the different but related model of the ``half-planar'' UIPT. Since these values coincide with the values $\dot{p}_c$, $\overline{p}_c$ that our paper identifies as thresholds for the behavior of the cluster of the origin in a critical Boltzmann triangulation, it is tempting to give a direct argument that also identifies these values with the percolation thresholds for the UIPT. 


\begin{proof}[Proof of Theorem \ref{thm:UIPT}]
 We perform the proof only in the case of site percolation, the arguments being exactly the same for bond percolation. 
Let $p>\dot{p}_c$ be fixed. 
Recall that by point 4 for Proposition \ref{prop:interface-site} we have $$\P{\vv(\dot{\mathfrak{C}}(p))\geq n}\sim c\, n^{-3/2}\, ,$$
for some $c\in (0,\infty)$. Since the Boltzmann triangulation $M$ itself is regular critical, its number of vertices satisfies a similar tail estimate
$$\P{\vv(M)\geq n}\sim C\, n^{-3/2}\, ,$$
for some $C\in (0,\infty)$. Since $\dot{\mathfrak{C}(p)}$ is a submap of $M$, the event $\{\vv(\dot{\mathfrak{C}}(p))\geq n\}$ is the same as 
$\{\vv(\dot{\mathfrak{C}}(p))\geq n, \vv(M)\geq n\}$. Hence, there exists $n_0$ such that for every $n\geq n_0$,
\begin{align*}
 0< \frac{c}{2C}&\leq \frac{\P{\vv(\dot{\mathfrak{C}}(p))\geq n}}{\P{\vv(M)\geq n}}
= \P{\vv(\dot{\mathfrak{C}}(p))\geq n\, \Big|\, \vv(M)\geq n}\, .
 \end{align*}
Therefore, for every $N>0$, and for $n\geq \max(n_0, N)$, it holds that $
 \mathbb{P}(\vv(\dot{\mathfrak{C}}(p))\geq N\, |\,  \vv(M)\geq n)\geq c/2C>0$. 
 Since the event $\vv(\dot{\mathfrak{C}}(p))\geq N$ is a local event (which depends only on the ball of radius $n$ around the root edge), we obtain by passing to the limit that $\mathbb{P}(\vv(\dot{\mathfrak{C}}_\infty(p))\geq N)\geq c/2C>0$. Letting $N\to\infty$ shows that $\dot{\mathfrak{C}}_\infty(p)$ is in fact infinite with positive probability, so that $\dot{p}_c(\mathrm{UIPT})\leq \dot{p}_c$. 

To show the other inequality, we prove that when $p< \dot{p}_{c}$ the size of the origin cluster in the UIPT has an exponential tail. We use absolute continuity relations between the UIPT $M_{\infty}$ and the critical Boltzmann triangulation $M$ as proved in \cite[Theorem 5 and Section 6.1]{CLGpeeling} or \cite[Proposition 7]{BCKscalingpeeling}: Given the critical Boltzmann triangulation $M$, there is a martingale $ (\mathcal{M}_{r})_{ r \geq 0} = ( \mathcal{M}_{r}(M))_{r \geq 0}$ depending only on the ball of radius $r$ 
such that for any positive function $F$ we have 
$$ \mathbb{E}[F(B_{r}(M_{\infty}))] = \mathbb{E}[ \mathcal{M}_{r}F(B_{r}(M))].$$
Clearly, this relation still holds if we consider percolated maps with the same parameter $p \in (0,1)$. Since the event on which the origin cluster has size at least $r$ is measurable with respect to the ball of radius $r$ we deduce that
 \begin{eqnarray*}\mathbb{P}( \vv( \dot{\mathfrak{C}}_{\infty}(p)) >r) &=& \mathbb{E}[ \mathcal{M}_{r} \mathbf{1}_{ \vv(\dot{\mathfrak{C}}(p)) >r}] \\
 &=& \mathbb{E}[ \mathcal{M}_{r} \mathbf{1}_{ \vv(\dot{\mathfrak{C}}(p)) >r} \mathbf{1}_{ \mathcal{M}_{r}\leq e^{cr}}] + \mathbb{E}[ \mathcal{M}_{r} \mathbf{1}_{ \vv(\dot{\mathfrak{C}}(p)) >r} \mathbf{1}_{ \mathcal{M}_{r}> e^{cr}}]\\
 & \leq & e^{cr} \mathbb{P}(\vv(\dot{\mathfrak{C}}(p)) >r) + \mathbb{E}[ \mathcal{M}_{r} \mathbf{1}_{ \mathcal{M}_{r}> e^{cr}}]. \end{eqnarray*}
From Proposition \ref{prop:interface-site} we know that when $p < \dot{p}_{c}$ then $\mathbb{P}(\vv(\dot{\mathfrak{C}}(p)) >r) \leq c_{1} e^{-c_{2}r}$ for some $c_{1},c_{2}>0$. It suffices to choose in the last display the constant $c=c_{2}/2$ to deduce that the first term in the last display decays exponentially. For the second term, we use the exact expression of $ \mathcal{M}_{r} \equiv \mathcal{M}_{r}( M)$ as given in \cite[Proposition 7]{BCKscalingpeeling} and deduce that for some $a >0$ we have $$ \mathcal{M}_{r}(M) \leq a \cdot \sum_{ \mathcal{C} \in \mathrm{Cycles}( \partial B_{r}( M))} \vv( \mathcal{C})^{3} \leq a \cdot \vv(\partial B_{r}( M))^{3} \leq a \cdot \vv( B_{r}( M))^{3}.$$
 Hence, using another time \cite[Theorem 5]{CLGpeeling} we get $$\mathbb{E}[ \mathcal{M}_{r} \mathbf{1}_{ \mathcal{M}_{r}> e^{cr}}] = \mathbb{E}[ \mathbf{1}_{ \mathcal{M}_{r}( M_{\infty})>e^{cr}}] \leq \mathbb{E}[ \mathbf{1}_{ a \cdot \vv(B_{r}(M_{\infty}))^{3}> e^{cr}}] \underset{ \mathrm{Markov\ ineq.}}{\leq} a^{1/3} \mathbb{E}[\vv(B_{r}(M_{\infty}))] e^{-cr/3}.$$ 
Moreover, it is well-known that $ \mathbb{E}[\vv(B_{r}(M_{\infty}))] = \Theta( r^{4})$ (see e.g.~\cite{Men16}) and so the last display indeed decays exponentially as $r \to \infty$. This completes the proof. \end{proof}

\subsection{$7/6$-stable map paradigm}\label{sec:stablemaps}

In \cite{LGM09}, Le Gall and the third author studied the scaling limits of bipartite $ \mathbf{q}$-Boltzmann maps where $ \mathbf{q}$ is critical admissible and where the disk partition function $ \mathsf{Disk}_{ \mathbf{q}}^{(\ell)}$ is orthodox with exponent $a \in (3/2;5/2)$. In particular they encode (using the Bouttier, Di Francesco, Guitter bijection \cite{BDFG04}) such random planar maps by some multitype Galton-Watson trees which are such that the offspring distribution is, in a sense, critical and in the domain of attraction of the spectrally positive stable law of parameter $$\alpha = a-1/2 \in (1,2).$$
In our case, by Proposition \ref{prop:interface-site} and \ref{prop:interface-bond} we should have $a = \frac{5}{3}$ hence $\alpha = \frac{7}{6}$.  But unfortunately the analysis of \cite{LGM09} is not directly applicable to our case because our maps are non necessarily bipartite. However, viewing this more as a technical problem than as a fundamental one\footnote{In our case we would be dealing with  Galton-Watson trees with  3 types of vertices, whereas the bipartite case treated in \cite{LGM09} has only 2 types of vertices.}, it is natural to perform a leap of faith and imagine that the large scale structure and critical exponents are the same as the ones found in \cite{LGM09}. This leads us to conjecture in particular that the (rescaled) critical percolation cluster $\mathfrak{C}$ conditioned on having $n$ vertices converges (in the Gromov-Hausdorff topology) toward the $7/6$-stable map defined in \cite{LGM09}\footnote{Recall however that the convergence in law of bipartite $ \mathbf{q}$-Boltzmann maps (in the Gromov-Hausdorff topology) was only proved to hold along subsequences in \cite{LGM09}.}.

We will now describe the anatomy of the critical percolation cluster $\mathfrak{C}$ (which can be either $\dot{ \mathfrak{C}}(\dot{p}_{c})$ or $ \overline{ \mathfrak{C}}( \overline{p}_{c})$) assuming that the results in  \cite{LGM09} extend naturally to non-bipartite maps.
First, with a probability of order $n^{-13/7}$ the cluster $\mathfrak{C}$ has total size (number of vertices) equal to $n$. On this event, the largest face in the cluster has a perimeter of order $n^{6/7}$ and the diameter of the cluster (for the graph distance restricted on the cluster) is of order $n^{3/7}$, see Figure~\ref{fig:criticalcluster} (left).
\begin{figure}[!h]
 \begin{center}
 \includegraphics[width=\linewidth]{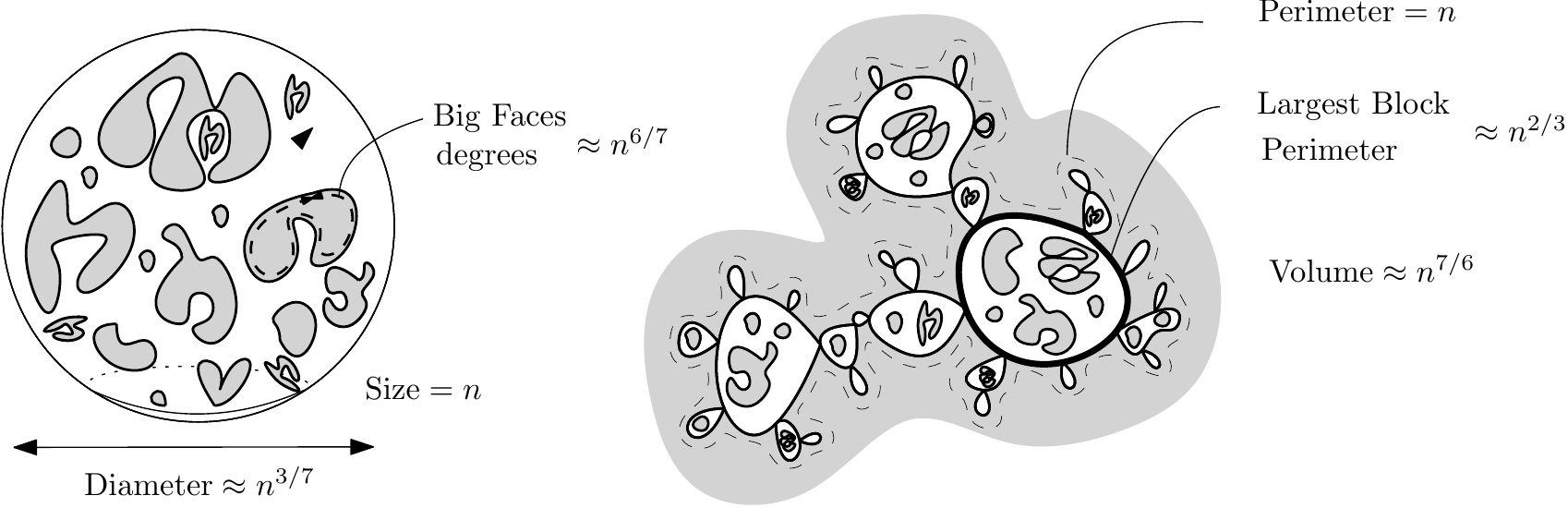}
 \caption{\label{fig:criticalcluster}Anatomy of a large critical percolation cluster on the (unlikely) event that it has size $n$ (left). On the right, the geometry of a large critical percolation with a boundary of perimeter $n$.}
 \end{center}
 \end{figure}

One can also wonder about the geometry of large critical cluster when we condition this cluster to have a root face of degree $n$ (note that by the above discussion, the cluster has size of order $n^{7/6}$ in this case). 
As we have seen in Proposition \ref{prop:interface-site} and \ref{prop:interface-bond} that the probability of this event decays like $n^{-10/3}$. On this event the external face is not at all a simple face but is folded on itself in the same manner as typical faces on $7/6$-stable maps. If one decomposes the cluster into blocks with simple boundary, then the tree structure of those blocks is described in the discrete setting by a critical random tree with offspring distribution in the domain of attraction of the $3/2$-stable law (see \cite{CKperco} for a rigorous treatment in the case of site-percolation on triangulations, and \cite{Richier17b} for a general treatment in stable maps). In particular the largest of these blocks has a perimeter of order $n^{2/3}$. One conjectures that the total size of such a block is already comparable to the total size of the cluster which is of order $n^{7/6}$. 

We are thus led to the following conjecture:
\begin{conj}\label{conj1} Consider a critical random Boltzmann triangulation $T^{(\ell)}$ of the $\ell$-gon and color in black its simple boundary. Then ,the cluster of the boundary  $ \mathfrak{C}^{(\ell)}$ of a critical (site or bond)  percolation on $T^{(\ell)}$ satisfies
$$ \vv( \mathfrak{C}^{(\ell)}) \approx \ell^{7/4}.$$
\end{conj}

In Conjecture \ref{conj1} and below, we use the notation $X_n\approx n^{\alpha}$ for a random variable $X_n$ to mean that for any $\epsilon>0$ the probability that $n^{\alpha-\epsilon}<X_n<n^{\alpha+\epsilon}$ tends to 1 as $n$ tends to infinity, and we say that $X_n$ is \emph{of order} $n^{\alpha}$ in this case.  The critical exponent of Conjecture \ref{conj1} may be used in conjunction with the recent work \cite{GMSS16} to compute the critical exponent of the size of the origin cluster in the UIPT. \\

Let us now examine, in each of the above pictures, the structure of the underlying \emph{triangulation} in which those large critical clusters are found. Let us condition again on the origin cluster $\mathfrak{C}$ having size $n$ (as in the left of Figure~\ref{fig:criticalcluster}). Of course, the random triangulation can be recovered by filling-in all the faces of the cluster  $\mathfrak{C}$ with the appropriate percolated triangulations with a boundary. As we already noticed above, a face of the cluster of degree $d$ is typically folded on itself and made of a tree of simple faces whose largest one is of degree $d^{2/3}$. 
Then, each of these simple faces must be filled-in by a Boltzmann triangulation with the appropriate perimeter. Since a Boltzmann triangulation with simple perimeter $\delta$ typically has size $\delta^{2}$, we deduce that the size of the sub-triangulation inserted in a face of large degree $d$ is expected to be of order $(d^{2/3})^2=d^{4/3}$ (because the size of this sub-triangulation should be comparable to the size of the  Boltzmann triangulation inserted in the largest simple boundary).   
Recalling that the maximal degree of the faces of  $\mathfrak{C}$ is of order $n^{6/7}$, we expect that the total size of the triangulation containing the large cluster $\mathfrak{C}$ of size $n$ has size of order $$ (n^{6/7})^{4/3}= n^{8/7}.$$
This is because, we the size of the triangulation containing $\mathfrak{C}$  should be comparable to the sub-triangulation contained in the largest face of $\mathfrak{C}$. 
We also conjecture that after proceeding to this filling operation, the initial cluster $\mathfrak{C}$ has a positive chance to be the largest cluster in the obtained percolated triangulation and in fact conjecture the following:
\begin{conj} Consider a uniform triangulation with $n$ faces and perform a critical (site or bond) percolation. Then, the largest black cluster $ \mathfrak{C}_{ \mathrm{max}}$ in the percolated triangulation satisfies $$ \vv( \mathfrak{C}_{ \mathrm{max}}) \approx n^{7/8}.$$
\end{conj}
Notice that the exponent $7/8$ conjectured above is in agreement with the KPZ relation and the known results for the largest cluster in critical site-percolation on $n \times n$ boxes in the regular triangular lattice in dimension $2$. Remark also that the two conjectures are linked to each other since a triangulation with boundary $ \ell$ has  roughly $\ell^{2}$ vertices and $(\ell^{{2}})^{7/8} = \ell^{7/4}$.



\bibliographystyle{abbrv}
\bibliography{bibli}

\end{document}